\documentclass[11pt]{article}
\usepackage{amssymb,amsmath}
\usepackage{amsthm,color}
\newtheorem{theorem}{Theorem}[section]
\newtheorem{proposition}{Proposition}[section]

\newtheorem{lemma}{Lemma}[section]

\DeclareMathOperator{\Exp}{Exp}

\DeclareMathOperator{\Ad}{Ad}
\DeclareMathOperator{\ad}{ad}

\DeclareMathOperator{\Det}{Det}
\DeclareMathOperator{\Id}{Id}
\DeclareMathOperator{\id}{id}
\DeclareMathOperator{\tr}{tr}
\DeclareMathOperator{\Tr}{Tr}
    
\DeclareMathOperator{\Lie}{Lie}

\DeclareMathOperator{\rank}{rank}

\DeclareMathOperator{\Herm}{Herm}

\DeclareMathOperator{\Cos}{Cos}
\DeclareMathOperator{\Sin}{Sin}

\begin{document}

\title{Intertwining operators for the generalized principal series on symmetric $R$-spaces}

\author{Jean-Louis Clerc}

\date{July 27, 2012}
\maketitle

\begin{abstract} Three questions about the intertwining operators for the generalized principal series on a symmetric $R$-space are solved : description of the functional kernel, both in the noncompact and in the compact picture, domain of convergence, meromorphic continuation. A large use is made of the theory of positive Jordan triple systems.  The meromorphic continuation of the intertwining integral is achieved via a Bernstein-Sato identity, and a precise description of the poles is obtained.
\end{abstract}
\footnotemark[0]{2000 Mathematics Subject Classification : 22E45, 43A80}

\section*{Introduction}

A \emph{symmetric $R$ space}, in a somewhat outdated terminology, is a real (generalized) \emph{flag manifold} $X=G/P$ ($G$ a real semi-simple Lie group, $P$ a parabolic subgroup of $G$), which is at the same time a \emph{Riemannian symmetric space}. Or vice-versa, $X$ is a compact Riemannian symmetric space with an "extra" action of a larger(noncompact) semisimple Lie group $G$ (the "big group") of diffeomorphisms of $X$ (for this point of view see \cite{na}, \cite{t}). 

The action of $G$ on $X$ gives raise to a family of representations $(\rho_\lambda)_{\lambda \in \mathbb C}$. From a geometric point of view, it is the action of $G$ on the $\lambda$-densities on $X$. From the semisimple harmonic analysis point of view, it is a special case of an induced representation from a character (depending on $\lambda$) of the parabolic subgroup $P$,  usually called the \emph{generalized principal series}. Associated to this family of representations is a family of operators $(J_\lambda)_{\lambda\in \mathbb C}$  which intertwine $\rho_\lambda$ with another representation induced from the opposite parabolic subgroup. The operators are defined for large values of $\Re \lambda$, and then extended meromorphically in the parameter $\lambda$, a special case of  a theorem  due to Vogan and Wallach (see \cite{vw}, \cite {w}). The theory however is not explicit enough to give complete information about the poles. 

Several authors studied specific examples : rank 1 compact symmetric spaces (\cite {jw}), Grassmannian manifolds (\cite{z2}, \cite {op}), Shilov boundaries of tube-type domains (\cite{oz}), spaces associated to real Jordan algebras (\cite{sah}, \cite{z1}), group case(\cite{jo}), mostly using the fact that $J_\lambda$ is a convolution operator on $X$, and computing the spectrum of the operator by use of the harmonic analysis on the compact Riemannian symmetric space $X$.

But to my best knowledge, this is the first time where all symmetric $R$-spaces are treated simultaneously and systematically. Following O. Loos (see \cite{l4}), symmetric $R$-spaces are approached through the theory of \emph{positive Jordan triple systems} (PJTS for short). 

Positive Hermitian Jordan triple systems  (PHJTS for short) and the associated compact Hermitian symmetric spaces play a special (double) role. Viewed with $\mathbb R$ as base-field, they are examples of PJTS. On the other hand, a PJTS (resp. a symmetric $R$-space) is a  \emph{real form} of a PHJTS (resp. of  a compact Hermitian symmetric space).

The  basic strategy is the following : first treat the case of PHJTS/compact Hermitian symmetric spaces with $\mathbb C$ as base-field, taking advantage of the holomorphic properties. Then, "restrict" to real forms to obtain the result for PJTS/symmetric $R$-spaces. Restriction should not be understood too strictly. In fact positive Jordan triples split in three families : the Hermitian case, the \emph{reduced} case and the \emph{non reduced} case, each family demanding its own interpretation of restriction. The Bernstein-Sato identity which is the basic tool for the meromorphic continuation comes in three "cousin" versions, all stemming from the same \emph{complex} Bernstein-Sato identity.

The example of the spheres, viewed as real projective quadrics may serve as an illustration. The sphere $S^n$ is a Riemannian symmetric space of rank $1$, which when viewed as a real projective quadric, becomes a $R$-space. The one dimensional sphere $S^1$ is a real form of the complex projective line $\mathbb P^1(\mathbb C)$, a compact Hermitian symmetric space of rank $1$ (reduced case). As $S^2 \simeq \mathbb P_1(\mathbb C)$, the two-dimensional sphere $S^2$ \emph{is} a compact Hermitian symmetric space of rank $1$, and it can be realized as a real form of $ \mathbb P^1(\mathbb C)\times \overline{\mathbb P^1(\mathbb C)}$, where $\overline{\mathbb P^1(\mathbb C)}$ is the complex projective line, but with the opposite complex structure. For $n\geq 3$,  the sphere $S^n$ is a real form of the complex quadric $\mathbb Q_n(\mathbb C)$, which is a compact Hermitian symmetric space of rank $2$ (non reduced case). 

We finish this introduction by a presentation of the different sections.  Section 1 addresses the three main questions precisely : 

$\bullet$ explicit form of the kernel of the intertwining operator, both in the compact and in the noncompact picture

$\bullet$ domain of convergence

$\bullet$ meromorphic continuation.

In Section 2  the relation between positive Jordan triple systems and symmetric $R$-spaces is exposed with some details, because the literature is not abundant on the subject, Section 3 presents the classification of  simple symmetric $R$-spaces,  Section 4 introduces the \emph{complex canonical kernel} for PHJTS, which is used in Section 5 to construct the \emph{canonical kernel} for a PJTS. Section 6 gives the realization of the canonical kernel in the compact picture, and gives the solution to the first question.  In section 7, the domain of convergence is determined, thus answering the second question. The rest of the paper (Sections 8--13) is devoted to the third question. The \emph{fundamental kernel} is introduced, first its complex version for a PHJTS (Section 8), then its real version for a PJTS (Section 9). Section 10 is the heart of the paper, as it produces  a Bernstein-Sato identity for the complex fundamental kernel of a PHJTS of tube-type, which is extended in Section 11 to general PHJTS. In Section 12, a Bernstein-Sato identity is proved for the fundamental kernel, first for PHJTS (subsection 12.1), then for a reduced PJTS (subsection 12.2), and eventually for a non-reduced PJTS (subsection 12.3).  Section 13 is devoted to the answer to the third question, namely the meromorphic continuation. Due to the length of the paper, possible applications to the study of the representations $\rho_\lambda$ are not considered.

\section {Symmetric $R$-space, generalized principal series and  intertwining operators}

A $R$-space $X$ is any quotient $X=G/P$, where $G$ is a real semi-simple Lie group (connected and with finite center) and $P$ a parabolic subgroup of $G$. Let $o=eP$ be the origin in $X$. Let $\sigma$ be a Cartan involution of $G$ and let $K=G^\sigma$ be the associated maximal compact subgroup of $G$. As $G=KP$, $X$ is homogeneous under the action of $K$. Let $K^o$ be the stabilizer of $o$ in $K$, so that $X\simeq K/K^o$. The space $X$ is said to be a \emph{ symmetric $R$-space} if there exists an involution $\theta$ of $K$ such that \[K^\theta_0\subset K^o \subset K^\theta\ ,\]
where $K^\theta= \{k\in K, \theta(k)=k\}$ and $K^\theta_0$ is the connected component of the neutral element of $K^\theta$.

Let $d\sigma$ be a $K$-invariant measure on $X$ (the normalisation will be specified later on). The group $G$ acts smoothly on $X$, and for $g\in G$ and $x\in X$, let $j(g,x)=d\sigma(gx)/d\sigma(x)$ be the Jacobian of $g$ at $x$.

Let $\lambda$ be a complex number. Then $G$ acts naturally on $\lambda$-densities on $X$. It is convenient to trivialize the density bundle by use of the measure $d\sigma$, so that the sections are identified with functions on $X$. For $g$ in $G$ define the operator $\rho_\lambda(g)$ on $\mathcal C^\infty(S)$ by 
\begin{equation}\label{rep}
\rho_\lambda(g) f (x) = j(g^{-1},x)^{\frac{1}{2}+\lambda} f(g^{-1}(x))
\end{equation}
$f\in \mathcal C^\infty(X)$. Then $\rho_\lambda$ is a  representation of $G$ on $\mathcal C^\infty(X)$, called the \emph{generalized (scalar) principal series} of $G$. 

The reason for the shift by $\frac{1}{2}$ in the parameter $\lambda$ is to make the following \emph{duality} property more symmetric.

\begin{proposition} Let $f, \varphi$ be in $\mathcal C^\infty(X)$
\[\int_V \big(\rho_\lambda(g) f\big)(x)\varphi(x) d\sigma(x) = \int_V f(y) \big(\rho_{-\lambda}(g^{-1})\varphi\big)(y) d\sigma(y)\ .
\]
\end{proposition}
The proof is by the change of variable $y=g^{-1}(x)$. The duality relation implies the unitarity of $\rho_\lambda$ on $L^2(X,d\sigma)$ when $\lambda$ is pure imaginary.

For $\mu$ a complex number, let $\rho_\mu^\sigma$ be the representation of $G$ defined by $\rho_\mu^\sigma(g)= \rho_\mu(\sigma(g))$. For generic values of $\lambda$, it is known (\cite{vw}) that there exists an intertwining operator between $\rho_\lambda$ and $\rho_{-\lambda}^\sigma$. The first goal of the present paper is to give a functional construction of this operator.

 Let $\widetilde c$ be a continuous everywhere nonnegative function on $X\times X$ satisfying the following \emph{covariance property} 
\begin{equation}\label{covtildec}
 \widetilde c\big(g(x), \sigma(g)(y)\big) = j(g,x)\widetilde c(x,y) j\big(g,\sigma(y)\big)
 \end{equation}
for all $x,y\in X\times X$ and $g\in G$. Let $J_\lambda$ be the operator (formally) defined by
\begin{equation}\label{jlambda}
J_\lambda(f)(x)  = \int_X \widetilde c(x,y)^{-\frac{1}{2} +\lambda} f(y) d \sigma(y)\ .
\end{equation}

\begin{proposition}
Let $\lambda\in \mathbb C$ such that
\begin{equation}\label{convergence}
\int_X \widetilde c(o,y)^{-\frac{1}{2} +\Re(\lambda)}d\sigma(y) <+\infty\ .
\end{equation}
Then the operator $J_\lambda$ defined by \eqref {jlambda} is a continuous operator on $\mathcal C^\infty(S)$ such that
\begin{equation}\label{intw}
J_\lambda\circ \rho_\lambda(g) = \rho_{-\lambda}^\sigma(g)\circ J_\lambda\ .
\end{equation}

\end{proposition}

\begin{proof} For $k\in K$, \eqref {covtildec} implies $\widetilde c(k(x), k(y)) = c(x,y)$ for any $x,y\in X$ and any $k\in K$. Hence $J_\lambda$ is a convolution operator on $X$, and the condition \eqref{convergence} guarantees that the operator $J_\lambda$ is a convolution operator with an integrable function on $X$. Hence $J_\lambda$ is a continuous operator on $\mathcal C^\infty(X)$. 

Let $f\in \mathcal C^\infty(S)$. Then 
\[ J_\lambda \big(\rho_\lambda(g) f\big)(x) = \int_X \widetilde c(x,y)^{-\frac{1}{2}+\lambda} j(g^{-1},y)^{\frac{1}{2}+\lambda} f(g^{-1}(y))d\sigma (y)\]
\[=\int_X \widetilde c(x,g(z))^{-\frac{1}{2}+\lambda}j(g,z)^{-\frac{1}{2}-\lambda} f(z) j(g,z) d\sigma(z)
\]
\[= j\big(\sigma(g),\sigma(g^{-1})(x)\big)^{-\frac{1}{2}+\lambda}\int_X \widetilde c(\sigma(g^{-1})(x), z)^{-\frac{1}{2}+\lambda}f(z) d\sigma(z)
\]
\[ =j\big(\sigma(g)^{-1},x\big)^{\frac{1}{2}-\lambda}I_\lambda f(g^{-1}(x))= \rho^\sigma_{-\lambda}(g)\big(J_\lambda f\big)(x)\ ,\]
where we used first the change of variable $y=g(z)$ and next, the covariance property  \eqref{covtildec} of the kernel $\widetilde c$.
\end{proof}

Three main questions are addressed in the present paper

$\bullet$ construct, as explicitely as possible a kernel $\widetilde c$, and its analogue in the noncompact picture

$\bullet$ determine the domain of convergence of the intertwining integral \eqref{convergence}

$\bullet$ prove the meromorphic continuation in the parameter $\lambda$ to $\mathbb C$ and determine the location of the poles.
 
\section{ Symmetric $R$-spaces and positive Jordan triple systems}

The main references for this section are \cite{l4} and \cite{l2}. For more developments on Jordan triple systems and Jordan pairs, see \cite{l3}. For the theory of compact Riemannian symmetric spaces, see \cite{hel}.

Let $X=G/P$ be a symmetric $R$-space. Keep notation introduced in the beginning of the previous section. The involution $\theta$ can be extended to an involution of $G$ (still denoted by $\theta$) and commuting with $\sigma$. Let
\[\mathfrak g = \Lie(G), \quad \mathfrak p = \Lie(P), \quad \mathfrak k =\Lie(K), \quad \mathfrak k^o = \Lie (K^o)\ .
\]
Denote by $\sigma$ and $\theta$ the associated involution of the Lie algebra $\mathfrak g$.
 
Let $V=T_oX$ be the tangent space of $X$ at $o$. The space $V$ carries two important algebraic structures : a structure of \emph{ Lie triple system}  (as does the tangent space to any Riemannian symmetric space) and a structure of \emph{positive Jordan triple system}. Let
\[\mathfrak k = \mathfrak k^o \oplus \mathfrak s\] be the eigenspace decomposition of the restriction of $\theta$ to $\mathfrak k$. Regarding $\mathfrak k$ as a Lie algebra of vector fields on $X$, the map $\xi\longmapsto \xi(o)$ is an isomorphism of vector spaces from $\mathfrak s$ onto $V$. For $v\in V$, let $\widetilde v$ be the unique vector field in $\mathfrak s$ such that $\widetilde v (o) = v$. Then $V$ is turned into a \emph{Lie triple system} (LTS for short) by
\[[u,v,w] = \big[\,[\widetilde u, \widetilde v], \widetilde w\big] (o)\ .
\]
Let $P^-=\sigma(P)$ be the opposite parabolic subgroup of $P$. Then $H=P^-\cap P^+$ is a Levi subgroup of $P$. The unipotent radical $N^+$ (resp. $N^-$) of $P$ (resp. of $P^-$) is Abelian, and  $N^- = \sigma(N^+)$. Let $\mathfrak h, \mathfrak n_+, \mathfrak n_-$ be the liea algebras respectively of $H, N^+, N^-$. There is a corresponding decomposition of the Lie algebra $\mathfrak g$ as
\[\mathfrak g = \mathfrak n_{-}\oplus \mathfrak h \oplus \mathfrak n_+
=\mathfrak n_{-}\oplus \mathfrak p\ .\]
Again regarding $\mathfrak g$ as a Lie algebra of vector fields on $X$, the map $\xi\longmapsto \xi(o)$ is an isomorphism of $\mathfrak n^-$ onto $V$. For $v\in V$, denote by $\widehat v$ the unique vector field in $\mathfrak n^-$ such that $\widehat v(o)=v$. Then $V$ is turned into a \emph{ positive Jordan triple system} (PJTS for short) by letting
\[\{u,v,w\} = -\frac{1}{2}\big[\,[\widehat u,\widehat v],\widehat w\big](o)\ .
\]

The triple product satisfies the fundamental identities for a Jordan  triple system 
\begin{equation}\label{PJTS1}
\{x,y,z\} = \{z,y,x\}
\end{equation}
\begin{equation}\label{PJTS2}
\{a,b,\{x,y,z\}\} = \{\{a,b,x\}, y,z\} -\{ x, \{b,a,y\},z\} + \{x,y, \{a,b,z\}\}
\end{equation}
for all $a,b,x,y,z\in V$.  Set 
\[ L(x,y)z=\{x,y,z\}, \quad Q(x)y = \{x,y,x\}\ .
\]
Define the \emph{trace form} of $V$  by
\[ (x,y) = \tr\big(L(x,y)\big)\ .\]
The trace form can be verified to be symmetric and positive definite, giving $V$ a structure of \emph{positive} Jordan triple system. In turn, $V$ becomes equipped with a Euclidean inner product. 

An important link between the two trilinear structures on $V$ is the identity
\begin{equation}\label{tritri}
[ u, v, w] = 2(-\{u,v,w\} + \{v,u,w\})\ .
\end{equation}
For a general discussion of the geometric significance of the relation between $LTS$ and $JTS$ structures, see \cite{b}.

Let $V$ be a PJTS. A linear operator $g$ on $V$ belongs to the \emph{structure group} $Str(V)$ if, for all $x,y,z$ in $V$,
\[ g\{x,y,z\} = \{ gx, {(g^{t})}^{-1} y, g z\}\ .\]
The structure group is a closed (hence Lie) subgroup of $GL(V)$. The Lie algebra $\mathfrak {str} (V)$ of $Str(V)$ is generated by the endomorphisms $L(u,v), u, v\in V$. Dilations belong to $Str(V)$, so that  $\id_V$ is in $\mathfrak {str}(V)$. Let  $Z=-\id_V$. Moreover, $L(u,v)^t = L(v,u)$ for any $u,v\in V$. This implies that $Str(V)$ is reductive in $GL(V)$.

A linear operator $k$ on $V$ is an \emph{automorphism} of $V$
 if, for all $x,y,z$ in $V$,
 \[ k\{x,y,z\} = \{ kx,ky,kz\}\ .\]
 The automorphisms of $V$ form a Lie group, denoted by $Aut(V)$ and $Aut(V) = Str(V)\cap O(V)$ . The Lie algebra of $\mathfrak{aut}(V)$ is characterized as $\{ X\in \mathfrak{str}(V), X^t=-X\}$, and it is generated by the endomorphisms $\big(L(u,v)-L(v,u)\big), u,v\in V$.

Conversely, it is possible to reconstruct the symmetric $R$-space from the PJTS structure on $V$. The process goes  first through the  \emph{Koecher-Kantor-Tits construction}.

\begin{proposition} Let $(V, \{\,.\,,\,.\,,\,.\,\})$ be a PJTS. 
Let \[\mathfrak g= \{(a,T,b), a\in V,  T\in \mathfrak{str}(V), b\in  V\}\]
and define the bracket of two elements $X=(a,T,b)$ and $X'=(a',T',b')$ by
\[
 [X, X'] = \big(Ta'-T'a, 2L(a',b)+[T,T'] -2L(a,b'), T'^t-T^t b'\big)\ .
\]

$i)$ $\mathfrak g$ is a  semi-simple Lie  algebra.

$ii)$ let
\[ \mathfrak n_{-} = \{ (a,0,0), a\in V\}\simeq V\]
\[  \mathfrak h =\{(0,T,0), T\in   \mathfrak {str}(V)\} \simeq \mathfrak {str}(V)\]
\[ \mathfrak n_+ = \{(0,0,b), b \}\simeq V\ . \]
Then $\mathfrak g= \mathfrak n^{-}\oplus \mathfrak h\oplus \mathfrak n^+$ is a $3$-graded Lie algebra, and the gradation is given by $\ad Z$.

$iii)$ the  map
\[ \sigma : (a,T,b)\longmapsto (b, -T^t, a) \]
is a Cartan involution of $\mathfrak g$.

$iv)$ the map
\[ \theta : (a,T,b)\longmapsto (-a,T,-b)\]
is an involution of $\mathfrak g$ which commutes with $\sigma$.

$v)$  identifying $(v,0,0)$ with $v$ for $v\in V$,
\[\forall \,a,b,c \in V,\qquad \{a,b,c\} = -\frac{1}{2}\big[\, [a,\sigma(b)],c\big] \ .
\]
\end{proposition}

Let 
\[ \mathfrak p = \mathfrak h \oplus \mathfrak n^+, \mathfrak k = \{ X\in \mathfrak g, \sigma X = X\}\ .\]
 Then $\mathfrak k$ splits under the action of $\theta$ as \[\mathfrak k = \mathfrak k^o\oplus \mathfrak s\] where
$\mathfrak k^o \simeq  \{ T\in \mathfrak {str}(V), T^t =-T\}= \mathfrak{aut}(V)$ and $\mathfrak s = \{ (a,0,a), a\in V\} \simeq V$.

The second step is to integrate these infinitesimal results. Let $G$ be the adjoint group of $\mathfrak g$, and denote by $\theta$ and $\sigma$ the involutions of $G$ corresponding to the  involutions of $\mathfrak g$. Let 

$\bullet $ $K=G^\sigma$  maximal compact subgroup (hence connected), 

$\bullet$ $P$ the normalizer of $\mathfrak p$

$\bullet$ $N^-=\exp(\mathfrak n_-)$  Abelian subgroup $\simeq \mathbb R^n$

$\bullet$ $H=G^\theta$ reductive subgroup with Lie algebra $\mathfrak h$

$\bullet$ $K^o = K\cap P$

Then $K^\theta_0\subset K^o\subset K^\theta$. Finally define $X$ as $G/P$. Then $X\simeq K/K^o$ is a symmetric $R$-space, associated to the PJTS $V$.

\begin{theorem} The map $V\longmapsto X$ establishes a one-to-one correspondance between (isomorphism classes of) PJTS and symmetric $R$-spaces.
\end{theorem}

See \cite {l2}, \cite{l4}. Let us mention that there is also a one-to-one correspondance with the $3$-graded real semisimple Lie algebras (see e.g. \cite{br}).

For $v$ in $V$, let $\overline n_v = \exp(v,0,0)\in N^-$. Then the  \[\kappa : V \longrightarrow X,\quad \kappa(v) =\overline n_v (o)\]
is a diffeomorphism on an open dense subset $\mathcal O$ of $X$. Hence we can transfer the action of $G$ on $X$ to a (not everywhere defined) action of $G$ on $V$ by
\[g(x) = \kappa^{-1}\big(g(\kappa(x) \big)\ .\]

It turns out to be a \emph{rational} action. By differentiation, there is a corresponding action of the Lie algebra $\mathfrak g$ by differential operators on $V$,  explicitly given by
\[(a,T,b) \longmapsto X (x) = (a+Tx+Q(x)b)\ \frac{\partial}{\partial x}
\]
By integration, it is possible, to a large extent, to find explicit expressions for the action of elements of $G$ on $V$ (see \cite{l4}). 

There is a spectral theory for PJTS, which will be used frequently. An element $c$ of $V$ is said to be a \emph{tripotent} if $\{c,c,c\} = c$. Then $L(c,c)$ is a symmetric operator on $V$, and its eigenvalues belong to $\{0, \frac{1}{2}, 1\}$, so that, with obvious notation \[V= V(c,1)\oplus V(c, \frac{1}{2}) \oplus V(c,0)\ ,\]
the so-called called  the \emph{ Peirce decomposition} of $V$ w.r.t. $c$.  When the tripotent $c$ is fixed it will be convenient to set
 \[V_2 = V(c,1), \quad V_1 = V(c,\frac{1}{2}),\quad V_0 = V(c,0)\ .\]

 The following property will be used frequently : for $i,j,k\in \{0,1,2\}$,
 \[\{ V_i,V_j,V_k\}\subset V_{i-j+k}\ ,
 \]
 where $V_\ell$ is assumed to be $0$ if $\ell \notin \{0,1,2\}$.

 The map $Q(c)$ is $0$ on $V(c,\frac{1}{2})\oplus V(c,0)$ and induces an involution on $V(c,1)$. Hence, the latter decomposes as
 \[V(c,1) = V^+(c,1)\oplus V^-(c,1)\ .\]

 Two tripotents $c,d$ are \emph{orthogonal} if $L(c,d) = 0$. If this is the case, then $L(c,c)$ and $L(d,d)$ commute, and $c+d$ is a tripotent.  
  
 A subspace $A$ of $V$ is said to be \emph{flat} if
\[\{A,A,A\} \subset A,\quad \{x,y,z\} = \{ y,x,z\}, \ \forall x,y,z\in A\ .
\]
If $(c_1,c_2,\dots, c_k)$ is a family of mutually orthogonal tripotents, then $A=\mathbb R c_1\oplus \mathbb R c_2\oplus \mathbb R c_k$ is a flat. 
 
There is a (partial) order on the tripotents : $c$ and $d$ being two tripotents, $c$ is smaller than $d$ if $d=c+e$, where $e$ is a tripotent, and $c$ and $e$ are orthogonal. A tripotent which is minimal ($0$ being removed) for this order is said to be \emph{primitive}. A  tripotent $c$ is primitive if and only if $V^+(c,1)=\mathbb R c$. A tripotent is said to be \emph{maximal} if it is maximal for this order. A tripotent $c$ is maximal if and only if $V(c,0)=\{0\}$.
 
 A \emph{Jordan frame} is a maximal set of mutually orthogonal primitive tripotents. Let $(c_1,c_2,\dots, c_r)$ be a Jordan frame.  The operators $L(c_j,c_j), 1\leq j\leq r$ are symmetric and commute to each other, so have a simultaneous decomposition in eigenspaces. More precisely, set
\[\begin{matrix} V_{jj}&& =&V(c_j,1) &j=1,\dots, r\\
V_{ij}&=V_{ji}& = &V(c_i,\frac{1}{2})\cap V(c_j,\frac{1}{2})&1\leq i<j\leq r\\
V_{j0}&=V_{0j}& =& V(c_j,\frac{1}{2})\cap \big( \cap_{k\neq j} V(c_k,0)\big)&j=1,\dots, r\ .
\end{matrix}
\]
Then the following \emph{Peirce decomposition} holds
\[V=\bigoplus _{0\leq i\leq j\leq r} V_{ij}\ .\]

The Peirce decomposition satisfies

\[\{ V_{ij}, V_{jk}, V_{kl}\} \subset V_{il}\ .
\]
and all other brackets are $0$.

Let $c=c_1+c_2+\dots c_r$. Then $c$ is a maximal tripotent, and the spectral decomposition of $V$ with respect to $c$ now reads
\[ V(c,1) = \bigoplus_{1\leq i\leq j\leq r} V_{ij},\quad V(c,\frac{1}{2}) = \bigoplus_{j=1}^r V_{oj},\quad V(c,0) = \{ 0\}\ .
\] 
The involution $Q(c)$ induces involutions on each $V_{ij}, 1\leq i\leq j \leq r$, giving the corresponding decompositions
\[V_{jj} = \mathbb R c_j \oplus V_{jj}^-, \quad V_{ij} = V^+_{ij} \oplus V^-_{ij}
\]

If $(c_1,c_2,\dots, c_r)$ is a Jordan frame, then the space $\mathbb R c_1\oplus \mathbb R c_2\oplus \dots\oplus \mathbb R c_r$ is a maximal flat of $V$. Conversely, given a maximal flat subspace $A$ of $V$, there exists a Jordan frame $(c_1,c_2,\dots, c_r)$ such that 
\[ A=\mathbb R c_1\oplus\mathbb R c_2\oplus \dots \oplus\mathbb R c_r\ .\]

 \begin{proposition} Two maximal flat spaces of $V$ are conjugate by some automorphism of $V$. Two Jordan frames of $V$ are conjugate up to order and signs.
 \end{proposition}
 
 The number of elements in a Jordan frame is called the \emph{rank} of $V$.
 
An \emph{ideal} of $V$ is a subspace $W$ such that $\{W,V,V\}\subset W$ and $\{V,W,V\}\subset W$. The PJTS $V$ is said to be \emph{simple} if it has no non trivial ideals. A PJTS can be decomposed as a sum of simple ideals in a unique way (up to order).

\begin{proposition} Let $V$ is a \emph{simple} PJTS of rank $r$. Let $(c_1,c_2,\dots, c_r)$ be a Jordan frame, and let $V=\bigoplus _{0\leq i\leq j\leq r} V_{ij}$ be the corresponding Peirce decomposition. Then

$\bullet$ for $1\leq i<j\leq r$, the dimension of $ V_{ij}$ (resp. $V_{ij}^+, V_{ij}^-)$ does not depend on $(i,j)$

$\bullet$ for $1\leq i\leq r$, the dimension of $V_{io}$ does not depend on $i$

$\bullet$ for $1\leq i\leq r$, the dimension of $V_{ii}$ does not depend on $i$
\end{proposition}
Consequently set
\[ a= \dim V_{ij},\quad a_+ = \dim V_{ij}^+,\quad a_- = \dim V_{ij}^-,\quad b = \dim V_{oi}, \quad c = \dim V_{ii}\ .
\]

The numbers $a,b,c$ are called the \emph{characteristic numbers} of $V$. 

Observe that $a=0$ if the rank of $V$ is $1$. Otherwise, $a$ is different from $0$. The characterisic number $c$ is always larger than $1$. If $c=1$ (which amounts to $V(c, 1) = \mathbb R c$ for one (equivalently any) primitive tripotent), the PJTS $V$ is said to be \emph{reduced}. The characteristic number $b$ may be $0$. In this case, $V$ is said to be \emph{of tube type} or \emph{of Jordan algebra type}\footnote{ For information on Jordan algebras, se \cite{fk})}. The reason for the first name is that the dual space to the associated symmetric $R$-space can be realized as a (real) tube-type domain (see \cite{l4}). The second name comes from the fact that, for a maximal tripotent $c$, the space $V=V(c,1)$ can be equipped with a structure of real semi-simple Jordan algebra $V_c$ by letting $x ._c y = \{ x,c,y\}$. The PJTS structure can be recovered from the Jordan algebra structure by the relation $Q(x) = P_c(x)\circ Q(c)$, where $P_c$ is the usual quadratic operator on the Jordan algebra $V_c$. We will be come to these points later.
 
The spectral theory of a PJTS $V$ can be connected to Lie triple system properties of $V\simeq\mathfrak s$, mainly through the identity \eqref{tritri}.

Let $(c_1,c_2,\dots, c_r)$ be a Jordan frame of $V$. Let \[A=\{x= \sum_{j=1}^r t_j c_j, t_j\in \mathbb R, 1\leq j \leq r\}\] be the corresponding maximal flat subspace of $V$. Then $A$ is a Cartan subspace of for the Lie triple structure on $V$, or otherwise said, 
\[\mathfrak a = \{ \widetilde x, x= \sum_{j=1}^r t_j c_j, t_j\in \mathbb R, 1\leq j \leq r\}\]
 is a Cartan subspace of $\mathfrak s$ . The restricted roots of the pair $(\mathfrak k, \mathfrak a)$ can be described as the linear forms $\lambda\neq 0$ on $\mathfrak a$ such that 
\[\mathfrak  s_\lambda = \{ X\in \mathfrak s, (\ad H)^2 X = -\lambda(H)^2 X, \forall a\in \mathfrak a\}\neq \{ 0\}\ .
\]

Let $x\in V$. Then (see \cite{l2} p. 221) \[(\ad \widehat x)^2\, \widehat y = \big(2(-L(x,x)+Q(x))y\big)\,\widehat { }\ .\]
Then, by elementary calculations, for $x= \sum_{j=1}^r t_j c_j$, the operators $L(x,x)$, $Q(x)$ and $\big(2(-L(x,x)+Q(x))y\big)$ can be described by the following matrix notation :
\[\begin{matrix} &V_{ii}^-& V_{ij}^+& V_{ij}^-& V_{i0}\\ \\
L(x,x)& t_j^2&\frac{1}{2}(t_i^2+ t_j^2)&\frac{1}{2}(t_i^2+ t_j^2)&\frac{1}{2} t_i^2\\ \\ Q(x) &-t_i^2&t_it_j&-t_it_j&0\\  \\2(L(x,x)-Q(x))&4t_i^2&(t_i-t_j)^2&(t_i+t_j)^2& t_i^2\\
\end{matrix}
\]
which gives both the roots and the root spaces. Denote by $\varepsilon_j$ the $j$-th coordinate on $\mathfrak a$ in the basis $(\widetilde c_1,\dots, \widetilde c_r)$.

\begin{proposition} The roots of the pair $(\mathfrak k, \mathfrak a)$ are 

$\bullet \pm 2 \varepsilon_i,1\leq i\leq r$ with multiplicity $c-1$

$\bullet \pm (\varepsilon_i-\varepsilon_j), 1\leq i<j\leq r$ with mutlipliciy $a_+$

$\bullet \pm (\varepsilon_i + \varepsilon_j), 1 \leq i<j\leq r$ with multiplicity $a_-$.

$\bullet \pm \varepsilon_i, 1\leq i\leq r$ with multiplicity $b$.

\end{proposition}

To this description of Cartan subspaces of $\mathfrak s$ corresponds a description of  maximal tori of $X$.

\begin{proposition} Let $(c_1,c_2,\dots,c_r)$ be a Jordan frame of $V$ and $A =Ê\sum_{j=1}^r \mathbb R\, c_j$ the corresponding maximal flat of $V$. Let $\Exp: A \longrightarrow X$ be the map defined by $\Exp v = \exp(\widetilde v) (o)$.

$i)$ the image $T = \Exp A$ is a maximal torus of $X$. 

$ii)$ the kernel of $\Exp$ is the lattice $\Gamma$ in $A$ given by
\[\Gamma = \sum_{j=1}^r \pi \mathbb Z \, c_j\ .
\]
\end{proposition}

For the proof see \cite{l2} p. 218. This shows that the Riemannian symmetric space $X$ has a \emph{cubic unit lattice}, and this property is characteristic of the symmetric $R$-spaces among the compact symmetric spaces (see \cite{l2}).

For $\boldsymbol \theta = (\theta_1,\theta_2,\dots, \theta_r)\in (\mathbb R/ \pi \mathbb Z)^r$, denote by $a_{\boldsymbol \theta}$ the element of $T$ defined by  
\[ a_{\boldsymbol \theta}= \Exp(\sum_{j=1}^r \theta_j c_j)\ .\]

A consequence of these results is the following integration formula on $X$, in the case where the PJTS $V$ is simple (see \cite {hel} Ch. I, theorem 5.10). 

\begin{proposition} For $a= \exp\big(\sum_{i=1} ^r t_j\widetilde {c_i}\big)$
Let \[D(a) = \Big\vert\prod_{i=1}^r (\sin 2t_j)^{c-1} \prod_{1\leq i<j\leq r} \sin(t_i-t_j)^{a_+}  \prod_{1\leq i<j\leq r}{\sin(t_i+t_j)}^{a_-} \prod_{i=1}^r {(\sin t_j)}^b\Big\vert\ .\]
Then, for any continuous function $f$ on $X$ which is invariant by $K^0$,
\begin{equation}\label{integration}
 \int_X f(x) d\sigma(x) =\int_T f(a(o)) D(a) da\ .
\end{equation}
\end{proposition}

\section{ Classification of simple PJTS and associated symmetric $R$-spaces}

As the set of roots has to be a root system $\Sigma$, the determination of the possible roots of the pair $(\mathfrak k, \mathfrak k^o)$ leaves five possibilities, each giving in turn some information on the values of the possible characteristic  numbers.
\medskip

$\bullet$ type $A_{r-1} \times A_0$  \quad $\Sigma = \{ \pm (\varepsilon_i-\varepsilon_j), 1\leq i\leq j\leq r\}$
\medskip

$\bullet$ type $B_r$  \quad $\Sigma= \{ \pm \varepsilon_i, \pm \varepsilon_j\pm \varepsilon_i, 1\leq i<j\leq r\}$
\medskip

$\bullet$ type $C_r$ \quad $\Sigma = \{ \pm 2 \varepsilon_i,\pm \varepsilon_j\pm \varepsilon_i, 1\leq i<j\leq r\}$
\medskip

$\bullet$ type $BC_r$\quad $\Sigma = \{ \pm \varepsilon_i,\pm 2 \varepsilon_i, \pm \varepsilon_i\pm \varepsilon_j, 1\leq i<j\leq r\}$
\medskip

$\bullet$ type $D_r$  \quad$\Sigma = \{ \pm \varepsilon_i\pm \varepsilon_j, 1\leq i<j\leq r\}$
\medskip

They correspond in the vocabulary of PJTS to 
\medskip

$\bullet$ Euclidean type  \quad $a_- = 0,\  b=0,\  c=1$
\footnote{The corresponding space $X$ are, as Riemannian symmetric spaces, not irreducible. In fact, $X$ is locally a product of a torus by an irreducible symmetric space, with root system of type $A_{r-1}$, hence the notation   for the root system. The typical exemple is $U(r)$ which is locally isomorphic to $U(1)\times SU(r)$. The PJTS of this category correspond to \emph{simple Euclidean Jordan algebras} (converted into triple systems) and the symmetric spaces are the \emph{Shilov boundaries} of the bounded symmetric domains of tube-type.}
\medskip

$\bullet$ reduced, non-tube type \quad $ b> 0,\  c=1 $
\medskip

$\bullet$ non-reduced, tube type \quad $ b=0,\  c>1$
\medskip

$\bullet$ non-reduced, not tube-type \quad $b>0,\  c>1$
\medskip

$\bullet$ reduced, tube-type, and not of Euclidean type \quad $ b=0, \ c=1$.
\medskip

The Weyl group for $A_{r-1}$ is the permutation group over $\{1,2,\dots,r\}$, for $B_r, BC_r$ and $C_r$, it is the signed permutation group and for $D_r$ it is the signed permutation group with en even number of minus signs.

As a byproduct of this rough classification, observe that for any simple PJTS, $a_-=a_+$, except for the Euclidean type (for which $a_-=0$), and (possibly) for the cases of type $D_2$, as these are the only cases where the Weyl group does not map $\varepsilon_1-\varepsilon_2$ to $\varepsilon_1+\varepsilon_2$. The latter situation occurs only for the PJTS $\mathbb R^{p,q}$, for $2\leq p<q$.
\bigskip

\centerline {Table 1}
\centerline {\bf Irreducible PHJTS and compact Hermitian symmetric spaces}

$$\vbox{\halign{#&#&#&#&#&#&#\cr
& &\quad$\mathbb V\quad$&\quad $\mathbb X$\quad&$\quad r$\quad&\quad$a_\mathbb C$\quad&\quad$b_\mathbb C$\quad\cr
\noalign{\vskip 2mm}
&$\quad s\geq 1$\quad&\quad$ \mathbb C^{1\times s}$\quad&\quad$\mathbb P_s(\mathbb C)\quad$ &\quad1\quad&\quad0\quad&\quad $s-1$\quad
\cr
\noalign{\vskip 2mm}
&${2\leq r\leq s}$&\quad $\mathbb C^{r\times s}$\quad&\quad $ Gr(r,\mathbb C^{r+s})$\quad&\quad $ r$\quad &\quad1\quad&\quad $s-r$\quad\cr
\noalign{\vskip 2mm}
&\quad${r\geq 3}$\quad &\quad $Asym(2r,\mathbb C)$\quad&\quad $SO(4r)/U(2r)$\quad&\quad$r$&\quad$2$&\quad$0$\cr
\noalign{\vskip 2mm}
&\quad${r\geq 2}$\quad &\quad$Asym({2r+1}, \mathbb C)\quad$&$\quad SO({4r+2})/U({2r+1})$\quad&\quad $r$\quad&\quad$2$\quad&$\quad2$\quad\cr
\noalign{\vskip 2mm}
&\quad${r\geq 2}$\quad&\quad$Sym(r,\mathbb C)$\quad&\quad$Sp(r)/U(r)$\quad&\quad$r$\quad&\quad$1$\quad&\quad$0$\quad\cr
\noalign{\vskip 2mm}
&\quad${n\geq 5}$\quad&\quad$\mathbb C^n$\quad&\quad$\mathbb Q^n(\mathbb C)$\quad&\quad$2$\quad&\quad$n-2$\quad&\quad$0$\quad\cr
\noalign{\vskip 2mm}
& &\quad$\mathbb O_\mathbb C^{1\times 2}$\quad&\quad$E_6/Spin(10).T$\quad&\quad$2$\quad&\quad$6$\quad&\quad$4$\quad\cr
\noalign{\vskip 2mm}
& &\quad$\Herm(2,\mathbb O_{\mathbb C})$\quad&\quad$E_7/E_6.T$\quad&\quad$3$\quad&\quad$8$\quad&\quad$0$\quad\cr
}}$$ 
\bigskip

\centerline{Table 2}
\centerline{\bf Simple reduced PJTS of Euclidean type}

$$\vbox{\halign{#&#&#&#&#&#\cr
& &\quad$ V\quad$&\quad $ X$\quad&$\quad r$\quad&\quad$a$\quad\cr
\noalign{\vskip 2mm}
&\quad&\quad$ \mathbb R\  $\quad&\quad$ S^1\quad$ &\quad1\quad&\quad0\quad\cr
\noalign{\vskip 2mm}
&\quad$r\geq 2$\quad&\quad$ Sym(r,\mathbb R) $\quad&\quad$SO(2r)/U(r)$\quad &\quad $r$\quad&\quad $1$\quad\cr
\noalign{\vskip 2mm}
&\quad$r\geq 2$\quad&\quad$ Herm(r,\mathbb C) $\quad&\quad$U(r)$\quad &\quad $r$\quad&\quad $2$\quad\cr
\noalign{\vskip 2mm}
&\quad$r\geq 2$\quad&\quad$ Herm(r,\mathbb H) $\quad&\quad$U(2r)/Sp(r)$\quad &\quad $r$\quad&\quad $4$\quad\cr
\noalign{\vskip 2mm}
&\quad$n\geq 4$\quad&\quad$\mathbb R^{1,n-1} $\quad&\quad$SO_0(n+2)/SO(n)\times SO(2)$\quad &\quad $2$\quad&\quad $n-2$\quad\cr
\noalign{\vskip 2mm}
&&\quad$Herm(3,\mathbb O)$\quad&\quad$E_6.T/F_4$\quad &\quad $8$\quad&\quad $n-2$\quad\cr
}}$$
\bigskip

\centerline{Table 3}

\centerline{\bf Simple reduced PJTS (not of Euclidean type)}

$$\vbox{\halign{#&#&#&#&#&#&#\cr
& &\quad$ V\quad$&\quad $ X$\quad&$\quad r$\quad&\quad$a$\quad&\quad $b$\quad\cr
\noalign{\vskip 2mm}
&\quad $2\leq s$ &\quad$ \mathbb R^{1\times s}\quad$&\quad $ \mathbb P_s(\mathbb R)$\quad&$\quad 1$\quad&\quad$0$\quad&\quad $s-1$\quad\cr
\noalign{\vskip 2mm}
&\quad$ 2\leq r\leq s$\quad &\quad$ \mathbb R^{r\times s}\quad$&\quad $ Gr(r,\mathbb R^{r+s})$\quad&$\quad r$\quad&\quad$2$\quad&\quad $s-r$\quad\cr
\noalign{\vskip 2mm}
&\quad$ 2\leq r$\quad &\quad$ Asym(2r,\mathbb R)\quad$&\quad $SO(2r) $\quad&$\quad r$\quad&\quad$2$\quad&\quad $0$\quad\cr
\noalign{\vskip 2mm}
&\quad$ 2\leq r$\quad &\quad$Asym(2r+1,\mathbb R)\quad$&\quad $SO({2r+1})$\quad &\quad $r$\quad&\quad$2$\quad&\quad $2$\quad\cr
\noalign{\vskip 2mm}
&\quad$ 2\leq p\leq q$\quad &\quad$\mathbb R^{p,q} $\quad&\quad $ Q^{p,q}(\mathbb R)$\quad&$\quad r$\quad&\quad$\Big\{\begin{matrix}{a_+=q-1}\\{a_- = p-1}\end{matrix}$\quad&\quad $0$\quad\cr
\noalign{\vskip 2mm}
&\quad$ $\quad &\quad$\mathbb O_s^{1\times 2} $\quad&\quad $ Gr(2,\mathbb H^4)/\mathbb Z_2$\quad&$\quad 2$\quad&\quad$6$\quad&\quad $4$\quad\cr
\noalign{\vskip 2mm}
&\quad$ $\quad &\quad$Herm(3,\mathbb O_s) $\quad&\quad $SU(8)/Sp(4).\mathbb Z_2$\quad&$\quad 3$\quad&\quad$4$\quad&\quad $1$\quad\cr
}}$$
\vfill \eject

\centerline{Table 4}

\centerline{\bf Simple non-reduced PJTS}

$$\vbox{\halign{#&#&#&#&#&#&#&#\cr
& &\quad$ V\quad$&\quad $ X$\quad&$\quad r$\quad&\quad$a$\quad&\quad $b$\quad&\quad c\quad\cr
\noalign{\vskip 2mm}
&\quad$2\leq s$\quad &\quad$ \mathbb H^s\quad$&\quad $ \mathbb P_s(\mathbb H)$\quad&$\quad 1$\quad&\quad$0$\quad&\quad $4(s-1)$\quad&\quad $ 4$\quad\cr
\noalign{\vskip 2mm}
&\quad$2\leq r\leq s$\quad &\quad $ \mathbb H^{r\times s}$\quad&\quad $ \mathbb G_r(\mathbb H^{r+s})$\quad&\quad $r$\quad &\quad $4$\quad&\quad $4(s-r)$\quad&\quad $4$\quad\cr
\noalign{\vskip 2mm}
&\quad$2\leq r$\quad &\quad $ Aherm(r,\mathbb H)$\quad&\quad $Sp(r)$\quad&\quad $r$\quad &\quad $4$\quad&\quad $0$\quad&\quad $3$\quad\cr
\noalign{\vskip 2mm}
&\quad$3\leq n$\quad &\quad $ \mathbb R^n$\quad&\quad $S^n$\quad&\quad $1$\quad &\quad $0$\quad&\quad $0$\quad&\quad $n$\quad\cr
\noalign{\vskip 2mm}
& &$\quad \mathbb O^{1\times 2}$\quad&\quad $\mathbb P^2(\mathbb O)$\quad&\quad $1$\quad &\quad $0$\quad&\quad $8$\quad&\quad $8$\quad\cr
\noalign{\vskip 2mm}
}}$$

The classification of PJTS can be found in \cite{fkklr}  Part II Table 4, or in \cite {l4}. See also \cite{n1}, \cite{n2}.

\section{Compact Hermitian symmetric spaces, positive Hermitian Jordan triple systems, the dual Bergman operator and the complex canonical kernel}

A well-known example of symmetric $R$-spaces is provided by the \emph{compact Hermitian symmetric spaces}. Their theory is well-known, they are often presented as the dual spaces of the noncompact Hermitian symmetric spaces. We refer to \cite{fkklr} 
Part III for a presentation of the main results. See also \cite{sa}.

We slighly modify our notation. The space $\mathbb X$ is then a complex manifold which is $\mathbb G/\mathbb P^+$ where $\mathbb G$ is a complex semi-simple Lie group and $\mathbb P^+$ a certain complex parabolic subgroup. The Cartan involution $\sigma$ is conjugate linear, whereas the involution $\theta$ is complex linear.The maximal compact subgroup is denoted by $U$ and $U^o = U\cap \mathbb P^+$. 

This category of spaces corresponds to PJTS admitting a complex structure, that is an operator $J$ on $V$ such that $J^2=-\id$ and
\[ J\{x,y,z\} = \{Jx,y,z\} = -\{x,Jy,z\}\ .\]

The space $V$ viewed now as a \emph{complex} vector space is denoted by $\mathbb V$ and called a \emph{positive Hermitian Jordan triple system} (PHJTS for short). In other words, a PHJTS is a complex vector space $\mathbb V$ with a map
$(x,y,z)\longmapsto \{x,y,z\}$
 which is $\mathbb C$-linear in $x$ and $z$, $\mathbb C$-conjugate linear in $y$ and satisfies the algebraic conditions \eqref{PJTS1} and \eqref{PJTS2}.

The \emph{trace form} of $\mathbb V$  given by
$(x,y) ={ \tr}_\mathbb C(L(x,y)$
 is a sesquilinear, and is positive-definite (as Hermitian form). If $L$ is a $\mathbb C$-linear endomorphism of $V$, denote by $L^*$ its adjoint with respect to the trace form. In particular, for any $x,y\in V$, $L(x,y)^* = L(y,x)$.

The group $Str(\mathbb V)$ is defined along the same line as before and is now a complex Lie subgroup of $GL(\mathbb V)$. The group of automorphisms $Aut(\mathbb V)$ is equal to $Str(\mathbb V)\cap U(\mathbb V)$.

The action of $\mathbb G$ on $\mathbb X$ can be transfered to a rational holomorphic action on $\mathbb V$ as before. For $g\in \mathbb G$ defined at $x\in \mathbb V$, we let $J(g,x)$ be the differential of $g$ at $x$. The differential $J(g,x)$ can be shown to belong to $Str(\mathbb V)$.

Let $\tau = \theta\circ \sigma$. The fixed points set of this third involution  of $\mathbb G$ is the group of holomorphic isometries of the noncompact Riemannian symmetric space \emph{dual} to $\mathbb X$. The \emph{Bergman operator} is defined  for $x,y\in \mathbb V$ by
\[B(x,y) = \id_V-2L(x,y)+Q(x)Q(y)\] 
This is a $\mathbb C$-linear endomorphism of $\mathbb V$, and, as a function of $x$ (resp. $y$) it is holomorphic (resp. conjugate holomorphic). The following proposition gives its most important property.

\begin{proposition} For $g\in \mathbb G$ and $x,y\in \mathbb V$ such that $g(x)$ and $\tau(g)(y)$ are defined,
\begin{equation}\label{covb}
B(g(x), \tau(g)(y))= J(g,x)\ B(x,y)\  J(\tau(g), y)^* \ .
\end{equation}
\end{proposition}
See \cite{sa} (ch II,  Lemma 5.2).

This formula is the holomorphic extension to $\mathbb G$ of the covariance property of the action of the isometry group of the bounded symmetric domain which is the dual Riemnnian symmetric space of $\mathbb X$. We need a version of this property for $\mathbb X$.

Let 
\begin{equation}C(x,y) = B(x,-y) = \id_V+2L(x,y)+Q(x)Q(y)
\end{equation}
which we call the \emph{dual Bergman operator} of $\mathbb V$.
\begin{proposition}
 For $g \in \mathbb G$ and $x,y\in \mathbb V$ such that $g(x)$ and $\sigma(g)(y)$ are defined,

\begin{equation}\label{covC}
C(g(x), \sigma(g)(y)) = J(g,x)\ C(x,y)\  J(\sigma(g),y)^*\ .
\end{equation}

\end{proposition}

\begin{proof}
Let $\iota = \exp(i\pi Z)$. Then $\Ad \iota = \theta$, so that the involution $\theta$ of $\mathbb G$ coincides with the inner automorphism $g\longmapsto \iota\circ g\circ \iota$. Moreover, for $x\in \mathbb V$, $\iota(x) = -x$. Hence, \eqref{covb} can be rewritten as

\[ B(g(x), -\sigma(g)(-y)) = J(g,x)B(x,y)J(\iota \circ \sigma(g)\circ \iota, y)^*\ .
\]
Now, by the chain rule,
\[J(\iota\circ \sigma(g)\circ \iota, y) = -\id_V\circ J(\sigma(g),-y)\circ (-\id_V) = J(\sigma(g),-y)\ .
\]
Now change $y$ to $-y$ to obtain \eqref{covC}.
\end{proof}

The \emph{complex canonical kernel} on $\mathbb V\times \mathbb V$ is defined by
\[ c(x,y) = {\det}_{\mathbb C} \,C(x,y)\ .\]

The complex canonical kernel is holomorphic in $x$ and antiholomorphic in $y$. It satisfies the covariance relation
\begin{equation}\label{covc}
c(g(x), \sigma(g) y) = j(g,x) c(x,y)\overline {j(\sigma(g),y)}\ ,
\end{equation}
where $j(g,x)$ is the (complex) \emph{Jacobian} of $g$ at $x$.

{\bf Example}. Let $V=\mathbb C$, with the Jordan triple product
\[\{x,y,z\} = x\,\overline y\,z\ .
\]
Then $\mathbb G = SL_2(\mathbb C)$ acts on $\mathbb C$ by
\[\Bigg(g=\begin{pmatrix} a&b\\c&d \end{pmatrix},x\Bigg)\longmapsto \frac{ax+b}{cx+d}\ .
\]
The Cartan involution $\sigma$ is given by
\[\sigma(\begin{pmatrix} a&b\\c&d \end{pmatrix}) = \begin{pmatrix} \overline d&-\overline c\\-\overline b&\overline a \end{pmatrix}
\]
The canonical kernel is given by $c(x,y) = 1+2x\overline y +x^2\overline y^2=(1+x\overline y)^2$ and
\[j(g,x) = \frac{1}{(cx+d)^2}, \quad j(\sigma(g), y) = \quad \frac{1}{(-\overline b y +\overline a)^2}
\]
The transformation formula \eqref{covc} now reads

\[ \Bigg(1+ \Big(\frac{ax+b}{cx+d}\Big)\Big(\frac{d\overline y +c}{-b\overline y +a}\Big)\Bigg)^2 = \frac{1}{(cx+d)^2}\,(1+x\overline y)^2\,\frac{1}{(-b\overline y +a)^2}\ \ .
\]

\section{The dual Bergman operator and the canonical kernel for a PJTS}

Let $\mathbb V$ be a PHJTS. An \emph{involution} $\alpha$ of $\mathbb V$ is, by definition, a conjugate linear map of $\mathbb V$ which is involutive ($\alpha\circ \alpha = \id)$ and satisfies
\[\{ \alpha x, \alpha y, \alpha z\} = \alpha\{x,y,z\}\ .
\]
Then $V=\{ x\in \mathbb V, \alpha(x) = x\}$ is easily seen to be a PJTS. The PJTS $V$ is said to be a \emph{real form} of the PHJTS $\mathbb V$. Now conversely, any PJTS in a canoncial way a real form of a PHJTS. In fact, let $V$ be a PJTS. Then let $\mathbb V = V\otimes_{\mathbb R} \mathbb C$ be its complexification and extend the Jordan product $\{x,y,z\}$ to $\mathbb V$ in a $\mathbb C$-linear way in $x$ and $z$ and in a conjugate linear way in $y$. Then $\mathbb V$ is easily seen to be a PHJTS, and the conjugation with respect to $V$ is an involution of $\mathbb V$. The space $\mathbb V$ is called the \emph{Hermitification} of $V$.

Let $V$ be a PJTS, let $\mathbb V$ be its Hermitification,  and let $^- : z\mapsto \overline z$ be the conjugation of $\mathbb V$ w.r.t. $V$. There corresponds an involution of $\mathbb G$, given by $g\longmapsto { ^-}\circ g\, \circ { ^-}$ and the fixed points set $G$ of this involution is a real form of $\mathbb G$. It is a real semisimple Lie group. Let $P=\mathbb P^+\cap G$ which is a parabolic subgroup of $G$. Then the quotient space $X=G/P$ is a real form of $\mathbb X =\mathbb G/\mathbb P^+$. The Cartan involution $\sigma$ of $\mathbb G$ restricts to a Cartan involution of $G$, and  $X = (G\cap U)/ (G\cap U^o)= K/K^o$ is realized as a compact Riemannian symmetric space.

 As for the case of PHJTS, define the \emph{dual Bergman operator} of $V$ as
\[C(x,y) = \id+2L(x,y) + Q(x)Q(y)\ .
\]

Then \eqref{covC} implies, for $x,y\in V$ and $g\in G$ such that $g$ is defined at $x$ and $\sigma(g)$ defined at $y$,
\[C\big(g(x),\sigma(g)(y)\big) = J(g,x) C(x,y) J\big(\sigma(g),y\big)^t\ .
\]

Define the \emph{canonical kernel} on $V\times V$ to be
\[c(x,y) =\vert \Det C(x,y)\vert\ . \] 
For $g\in G$ and $x\in V$ and $g$ defined at $x$, let  $j(g,x) =\vert  \Det J(g,x) \vert$ be the Jacobian of $x$ at $g$.

\begin{proposition} Let $g\in G$, and $x,y\in V$ such that $g$ is defined at $x$ and $\sigma(g)$ is defined at $y$. Then
\begin{equation}\label{ccov}
c(g(x), \sigma(g) (y)) =  j(g,x) c(x,y)  j(\sigma(g), y)\ .
\end{equation}
\end{proposition}
This property follows from the covariance property of the dual Bergman operator.

\begin{proposition} Let $V$ be a simple PJTS of rank $r$ and let $(c_1,c_2,\dots, c_r)$ be a Jordan frame of $V$. Let $x=\sum_{j=1}^r x_jc_j$, and set $x_0=0$. 
\medskip

$i)$ for $0\leq i\leq j\leq r$ and $y_{ij}\in V_{ij}$
\[C(x,x) y_{ij} = (1+x_i^2)(1+x_j^2) y_{ij}\ .\]

$ii)$ 

\begin{equation}
c(x,x) =\big( \prod_{j=1}^r (1+x_j^2)\big)^p 
\end{equation}

where $p= (r-1)a +b+2c$ is the \emph{genus} of $V$.

\end{proposition}

\begin{proof}The first statement is a consequence of the following identities, for $x= \sum_{j=1}^r x_j c_j$ and $y_{ij}\in V_{ij}, 0\leq i\leq j\leq r$
\[C(x,x) = \id +2L(x,x)+Q(x)^2\]
 \[L(x,x) = \sum_{j=1}^r x_j^2 L(c_j,c_j)\]
 \[ Q(x)y_{ij} = x_ix_j Q(c)y_{ij}\]
where $c=\sum_{j=1}^r c_j$, and the fact that $Q(c)^2$ is the projection on $V(c,1)$. As $C(x,x)$ is scalar on each $V_{ij}$, its determinant is easily computed.
\end{proof}

The canonical kernel satisfies the \emph{symmetry property}
\begin{equation}
c(x,y) = c(y,x)
\end{equation}
for all $x,y\in V$, which is a consequence of the symmetry property of the (dual) Bergman opeartor
\[C(y,x) = C(x,y)^t\ .\]

\begin{proposition}\label{cV2}
 Let $c$ be a maximal tripotent in $V$, and let $V=V_2\oplus V_1$ be the associated Peirce decomposition.  Then the following identity holds :
\begin{equation}\label{V2}
c(x_2+x_1,y_2)= c(x_2,y_2).
\end{equation}
for any $x_1\in V_1, x_2\in V_2$ and $y_2\in V_2$.
\end{proposition}

\begin{proof} The linear operators on $V$ will be denoted in matrix form with respect to the decomposition $V=V_2+V_1$. Let $x=x_2+x_1$ be an arbitrary element of $V$. Then $L(x,y_2)$ has the following form :
\[L(x,y_2) = \begin{pmatrix}\{x_2,y_2,.\}&0\\\{x_1,y_2,.\}&\{x_2,y_2,.\} \end{pmatrix}\ .
\]
Similarly,
\[Q(y_2) = \begin{pmatrix}\{y_2,.,y_2\}&0\\0&0 \end{pmatrix}
\]
and
\[ Q(x) = \begin{pmatrix}\{x_2,.,x_2\}&2\{x_1,.,x_2\}\\ & \\2\{x_1,.,x_2\}&\{x_1,.,x_1\}\end{pmatrix}
\]
so that
\[C(x,y_2) = \begin{pmatrix}1_{W_2} +2L(x_2,y_2)_{\vert V_2}+Q(x_2)_{\vert V_2}Q(y_2)_{\vert V_2}&0\\ & \\&\\\star&1_{W_1} +2\{x_2,y_2,.\}\end{pmatrix}\ .
\]
Hence,
\[ c(x,y_2) = c(x_2,y_2)\ . \]
\end{proof}

\section{The canonical kernel in the compact picture}

We use freely of the notation introduced so far. Let $V$ be a simple PJTS. In this section, we transfer the previous results in the compact picture, using the map $\kappa$.  Some preliminary results will be needed.

\begin{proposition}\label{invm} The measure on $V$ define by
\[ f\longmapsto \int_V f(x) \,c(x,x)^{-\frac{1}{2}} dx\]
is invariant under the action of $K$.
\end{proposition}

\begin{proof} Let $k$ be in $K$. Then $\sigma(k) = k$, so that, at any point $x$ where $k(x)$ is defined,
\[c(k(x),k(x))= j(k,x)^2 c(x,x)\ .
\]
Now 
\[ \int_V f(k^{-1}(x)) c(x,x)^{-\frac{1}{2}} dx = \int_V f(y) c(k(y), k(y))^{-\frac{1}{2}} j(k,y) dy\]
\[ = \int_V f(y) c(y,y)^{-\frac{1}{2}} dy\ .
\]
\end{proof}

This result can be used to normalize the $K$-invariant measure $d\sigma$  on $X$, by the condition that
\[ \int_X f(x) d\sigma = \int_V f ( \kappa (v)) c(v,v)^{-\frac{1}{2}}dv\ .\]
In turn, it implies, that for this normalization of the measure $d\sigma$
\begin{equation}\label{jack}
 j(\kappa,v)  = \frac{ d\sigma\big(\kappa(v)\big)}{ dv} = c(v,v)^{-\frac{1}{2}}\ . 
\end{equation}

\begin{proposition} Let $(c_1,c_2,\dots, c_r)$ be a Jordan frame of $V$. For $\boldsymbol {\theta} = (\theta_1,\theta_2, \dots, \theta_r) \in (\mathbb R/\pi\mathbb Z)^r$, 
\begin{equation}
\kappa(\sum_{j=1}^r \tan \theta_j c_j) = \Exp(\sum_{j=1}^r \theta_j  c_j)\end{equation}
\end{proposition}

For a proof, see \cite{l2}.

The following theorem answers the question we addressed at the end of section 1.

\begin{theorem}\label{ctildex}
There exists a unique everywhere nonnegative continuous function $\widetilde c$ on $X\times X$ such that
\medskip

$i)$ $\widetilde c \big((o, a_\theta ( o)\big) =\big( \prod_{j=1}^r  \cos^2 \theta_j\big)^{\frac{p}{2}}$,
for all $\theta \in (\mathbb R/\pi \mathbb Z)^r$.
\medskip

$ii)$ for all $s,t\in X$ and $g\in G$
\[ \widetilde c(g(s), \sigma(g)(t)) = j(g,s)\,\widetilde c(s,t)\,j(\sigma(g),t)\ .\]

\end{theorem}
 
The following proposition is a preparation for the proof of the theorem.

\begin{proposition}\label{ktildex}
 There exists a unique smooth function $\widetilde k$ on $X\times X$ such that
\medskip

$i)$ $\widetilde k (o, a_\theta) = \prod_{j=1}^r  \cos^2 \theta_j$,
for all $a_\theta \in T$.
\medskip

$ii)$ $\widetilde k(k(s)), k(t)) = \widetilde k(s,t)$, for all $s,t\in X$ and $k\in K$.

\end{proposition}

\begin{proof}

The proof uses several results of harmonic analysis on a compact symmetric space. The function $\ell : a_\theta \longmapsto  \prod_{j=1}^r  \cos^2 \theta_j$ is a smooth function on the torus $T\simeq (\mathbb R/\pi \mathbb Z)^r$, which is moreover invariant by the Weyl group\footnote{Recall that the Weyl group consists in permutations of $\{1,2,\dots, r\}$ and possibly sign changes}. Hence there exists a unique smooth function (still denoted by $\ell$) on $X$ which is invariant by $K^o$ and coincides on the torus with $\ell$. Now, there exists a unique smooth function $\widetilde k$ on $X\times X$ such that $\widetilde k(k(s),k(t))= \widetilde k(s,t)$ and $k(o,t) = \ell(t)$ for any $k\in K$ and $s,t\in X$.
\end{proof}

Set $\widetilde c(s,t) = \widetilde k (s,t) ^{\frac{p}{2}}$. The function $\widetilde c$ is continuous and everywhere nonnegative, as this was already true for $\widetilde k$. The covariance relation for $\widetilde c$ under the action of $G$ is obtained through the comparaison of $\widetilde c$ with $c$.

Let $\mathcal O = \kappa(V)$, which is a dense open subset of $X$.

\begin{proposition} Let $s=\kappa(x)$ and $t=\kappa(y)$ be in $\mathcal O$. Then
\begin{equation}\label{ctildec}
\widetilde c(s,t) =c(x,x)^{-\frac{1}{2}} c(x,y) c(y,y)^{-\frac{1}{2}}
\end{equation}
\end{proposition}

\begin{proof} Observe that both sides are invariant by the diagonal action of $K$ (at least where defined). Hence it is enough to verify the relation when $t=o$, which amounts to $y=0$, and $s=a_\theta$ in $T\cap \mathcal O$. The corresponding value of $x$ is $x=\sum_{j=1} \tan \theta_j c_j$, as $\kappa(x) = a_\theta $. Now $c(x,0)=1$, whereas
$c(x,x) =\big( \prod_{j=1}^r (1+\tan^2 \theta_j)\big)^p$. Thus \eqref{ctildec} follows. 
\end{proof}

\begin{proposition} The function $\widetilde c$ satisfies the following covariance property 

\[ \widetilde c(g(s), \sigma(g)(t)) = j(g,s)\,\widetilde c(s,t)\,j(\sigma(g),t)\ .\]
for all $s,t\in X$ and $g\in G$,
\end{proposition}

\begin{proof} For $g\in G$ and $s,t$ in $\mathcal O$ such that $g(s)$ and $g(t)$ are in $\mathcal O$, routine calculation starting from the covariance property satsified by $c$ and using \eqref{ctildec} and \eqref{jack}. The full covariance property is then obtained by using the continuity of $\widetilde c$ on $X\times X$.
\end{proof}

The proof of Theorem \ref{ctildex} is then complete. In fact the uniqueness statement follows form the same argument we gave for the uniqueness of $\widetilde k$ in the proof of Proposition \ref{ktildex}.

\section{The domain of convergence for the intertwining integral}

\begin{theorem}\label{absconv} 
Let $X$ be a simple symmetric $R$-space, of rank $r$ and characteristic numbers $r,a=a_++a_-, b,c$ and genus $p= (r-1)a+b+2c$. Let
\[c(\lambda) = \int_X \widetilde c(o,y)^{-\frac{1}{2}+\lambda} d\sigma(y)\ .
\]
The integral $c(\lambda)$ is absolutely convergent if and only if
\[\Re (\lambda) > \frac{1}{2} -\frac{c}{p} \ .
\]
\end{theorem}

\begin{proof}By use of the integration formula \eqref{integration}, the integral to be checked for convergence is equal to
\[\int_{[0,\pi]^r} \big(\prod_{j=1}^r \cos^2 \theta_j\big)^{\frac{p}{2}(-\frac{1}{2} +\lambda)} \Big\vert\prod_{i=1}^r (\sin 2\theta_j)^{c-1} \prod_{1\leq i<j\leq r} \sin(\theta_i-\theta_j)^{a_+} \dots \]\[\dots\prod_{1\leq i<j\leq r}{\sin(\theta_i+\theta_j)}^{a_-} \prod_{i=1}^r {(\sin \theta_j)}^b\Big\vert\ \, d\theta_1\dots d\theta_r\ .
\]

Let first assume that $a_+ = a_-=\frac{a}{2}$.   Using $\sin (\theta+\varphi) \sin (\theta-\varphi) = \cos^2 \theta -\cos^2 \varphi$, and $\sin 2 \theta =2 \sin \theta \cos \theta$,  the integral $c(\lambda)$ to be checked for convergence is equal to
\[\int_{[0,\pi]^r} \big\vert\prod_{j=1}^r \cos \theta_j\big\vert^{p(-\frac{1}{2} +\lambda) +c-1}\big(\prod_{j=1}^r \sin  \theta_j\big) ^{ b+ c-1} \big\vert\prod_{1\leq i<j\leq r} (\cos^2 \theta_i-\cos^2 \theta_j)\big\vert^{\frac{a}{2}} d\theta_1\dots d\theta_r
\]
The integrand is invariant under the changes $\theta_j\mapsto \pi-\theta_j$, so that we may integrate on $[0,\frac{\pi}{2}]^r$ and use the change of variables $u_j = \cos^2 \theta_j$ to get (up to a constant)
\[\int_{[0,1]^r} \prod_{j=1}^r u_j^{\frac{1}{2}(p(-\frac{1}{2} + \lambda) + c)-1}
\prod_{j=1}^r (1-u_j)^{\frac{1}{2}(c+b)-1} \prod_{1\leq i<j\leq r}\vert u_i-u_j\vert^{\frac{a}{2}} du_1\dots du_r
\]

This integral is a special case of the celebrated \emph{ Selberg's integrals} (see e.g. \cite{a} Theorem 8.1.1)
\[S_r(\alpha, \beta, \gamma) := \int_0^1\dots \int_0^1 \prod_{i=1}^r t_i^{\alpha-1} (1-t_i)^{\beta-1} \prod_{1\leq i<j\leq r} \vert t_i-t_j\vert^{2\gamma} dt_1\dots dt_r\ .
\]
The conditions for absolute convergence are
\[\Re(\alpha) >0,\quad \Re(\beta)>0, \quad \Re(\gamma)> -\min \{\frac{1}{r}, \frac{\Re(\alpha)}{r-1}, \frac{\Re(\beta)}{r-1}\}\ .
\]
In the case at hand, the conditions amount to 
\[\Re \big({\frac{1}{2}(p(-\frac{1}{2} + \lambda) + c)\big)}>0
\]
equivalent to
\[\Re (\lambda) > \frac{1}{2} - \frac{ c}{p}\ .
\]
The constant $c(\lambda) = J_\lambda 1$ plays an important role for further analysis of the intertwining operators and the representations $\rho_\lambda$. Although we won't use the result, it is worth to give the value of $c(\lambda)$ :

\[c(\lambda) = C \prod_{j=1}^r \frac{ \Gamma\big((j-1)\frac{a}{4}-\frac{(r-1)a+b}{4} +\frac{p}{2}\lambda\big)}{\Gamma\big((j-1) \frac{a}{4}+ \frac{b+2c}{4}+\frac{p}{2}\lambda\big)}\ ,
\]
where $C$ is an explicit  constant which depends only on $V$.
\bigskip

Let now $V$ be of Euclidean type, so that $a=a_+, a_- = 0, b=0$, $c=1$, and $p=(r-1)a+2$. The integral to be checked for convergence now reads
\[c(\lambda) = \int_{[0,\pi]^r} \big\vert\prod_{j=1}^r \cos \theta_j\big\vert^{p(-\frac{1}{2} +\lambda) } \big\vert\prod_{1\leq i<j\leq r} (\sin (\theta_i-\theta_j)\big\vert^a d\theta_1\dots d\theta_r\ .
\]
As the function to be integrated is periodic of period $\pi$ w.r.t. to each variable, the integral can be taken over $[-\frac{\pi}{2}, \frac{\pi}{2}]^r$. Then let $t_j = \tan \theta_j$, and use
\[ \sin^2(\theta_i - \theta_j)= \frac{ (t_i-t_j)^2}{(1+t_i^2)(1+t_j^2)}\]
to get
\[c(\lambda) = \int_{-\infty}^{+\infty}\dots\int_{-\infty}^{+\infty} \big(\prod_{j=1}^r (1+t_j^2)\big)^{-\frac{p}{2}(\frac{1}{2} +\lambda)} \big\vert\prod_{1\leq i<j\leq r} (t_i-t_j)\big\vert^a d t_1 d t_2\dots dt_r\ .
\]

This is again a special case for another version of the Selberg integral (in fact it is the original Selberg formulation, see \cite{fw}), which can be seen to converge if an only if $\Re(\lambda)>\frac{1}{2}- \frac{1}{p}$ and is equal to

\[c \ \prod_{j=0}^{r-1} \frac{ \Gamma(p(\frac{1}{2}+\lambda) -1 +(r+j-1)\frac{a}{2})}{\Gamma\big(\frac{p}{2} (\frac{1}{2} +\lambda) -j\frac{a}{2}\big)^2},
\]
where $c$ is an explicit  constant depending only on $V$.
\bigskip

From the classification, it remains now to consider the case of the PJTS $V=\mathbb R^{p,q}$, where $2\leq p< q$. For this case, $r=2, a_+=q-1, a_-=p-1, b=0$ and $c=1$, the genus is equal to $2+(q-1)+(p-1)=p+q=n$. So the integral to be considered is
\[\int_0^\pi\int_0^\pi \vert\sin(\theta_1-\theta_2)\vert^{q-1}\vert \sin(\theta_1+\theta_2)\vert^{p-1} \vert\cos \theta_1 \cos \theta_2\vert^{n(-\frac{1}{2}+\lambda)} d\theta_1 d\theta_2\ .
\]
The integral may become divergent because of the singularities near $\theta_1 = \frac{\pi}{2} $ or near $\theta_2 =  \frac{\pi}{2}$. This  forces the conditions $n(-\frac{1}{2}+\Re(\lambda) > -1$, i.e. $\Re(\lambda) > \frac{1}{2} -\frac{1}{n}$. Now if the two conditions are satisfied, it remains to check the convergence for both $\theta_1$ and $\theta_2$ close to $\frac{\pi}{2}$. Making the change of variables $\varphi_i = \theta_i-\frac{\pi}{2}$ for $i=1,2$, and using $\sin x \simeq x$ for $x$ close to $0$ leads to discussing the convergence near $(0,0)$ of the integral
\[\int \int \vert \varphi_1-\varphi_2\vert^{q-1} \vert \varphi_1+\varphi_2\vert^{p-1} \vert \varphi_1\varphi_2\vert^{n(-\frac{1}{2}+\Re(\lambda))} d\varphi_1 d\varphi_2\ .
\]
Using polar coordinates near $(0,0)$, this requires for convergence the (possibly) supplementary condition that
\[(q-1)+(p-1)+n(-1+2 \Re(\lambda))+1>-1
\]
which amounts to $\Re(\lambda)>0$. Hence the integral converges if and only if $\Re(\lambda) > \frac{1}{2} -\frac{1}{n}$.
\end{proof}

\section{The complex fundamental kernel for a PHJTS}

Let $\mathbb V$ be a simple PHJTS. For the bounded domain associated to $\mathbb V$, it is well known that the Bergman kernel is another polynomial raised at the power $g$, where $g$ is the genus of $\mathbb V$ (see e.g. \cite{fkklr} Part V Proposition VI.3.6). Similarly in our setting, by changing $y$ to $-y$, the complex canonical kernel $k(x,y)$ is equal to $h(x,y)^g$, where $h(x,y)$ is another polynomial on $\mathbb V \times \mathbb V$ and $g$ is genus of $\mathbb V$. We introduce the polynomial $h$ via the (complex) generic minimal polynomial, although there is a more direct way (see \cite{fkklr} Part III). But this approach will be useful in the the next section, where we will prove a similar result is for  the canonical kernel of a PJTS. This is an important result, specially towards proving a Bernstein-Sato identity. 

First recall the theory of the \emph{generic minimal polynomial} on a PJTS. Let $\mathbb K = \mathbb R$ or $\mathbb C$. Let $V$ be a PJTS (if $\mathbb K = \mathbb R$) or a PHJTS (if $\mathbb K= \mathbb C$). Let $y$ be in $V$, and define the Jordan product $x._y z$ on $V$ to be 
\[ x._yz = \{ x,y,z\}\ .\]
The space $V$ with this product is a Jordan algebra over $\mathbb K$, denoted by $V^{(y)}$.
The successive powers of the element $x$ for this Jordan product are denoted by $x^{(2,y)}, \dots, x^{(k,y)},\dots$. More generally, if $p$ is a polynomial with vanishing constant term, then $p(x)$ is well defined in $V^{(y)}$. Let 
$\rho_y(x) $ be the supremum of the values of $k$ for which $x,x^{(2,y)},\dots, x^{(k,y)}$ are linearly independant, and let $\rho = \sup_{(x,y)\in V} \rho_y(x)$ ($\rho$ is called the \emph{absolute rank} of $V$). For $x,y\in V\times V$, and $k$ an integer,  let $\beta_k(x,y)$ be the element of the exterior tensor product $\Lambda^k V$ defined by
\begin{equation}\label{betak}
\beta_k(x,y) = x\wedge x^{(2,y)} \wedge \dots \wedge x^{(k,y)}\ .
\end{equation}
Clearly $\beta_k$ is polynomial in $x$ and $y$. For $k\geq \rho+1$, $\beta_k$ vanishes identically on $ V\times  V$, whereas $\beta_\rho$ does not vanish identically on $V\times V$. A couple $(x,y)\in V\times V$ is said to be \emph{regular} if $\beta_\rho(x,y)\neq 0$.  Regulars elements form a Zariski open dense set in $V\times V$. For $(x,y)$ regular, let $m(T)=m(T,x,y)$ be the monic generator of the ideal in $\mathbb K[T]$ of polynomials $p$ such that  the polynomial $Tp(T)$ annihilates $x$ in $V^{(y)}$.

\begin{proposition} For $1\leq j\leq \rho$, there exist (unique) polynomials $m_j$ on $\mathbb K\times  V\times V$ such that, for any regular pair $(x,y)$ in $V\times V$
\begin{equation}\label{mmj}
m(T,x,y) = T^\rho +\sum_{j=1}^\rho (-1)^j m_j(x,y) T^{\rho-j}\ .
\end{equation}

The polynomials $m_j$ are homogeneous of bidegree $(j,j)$ in $(x,y)$. \end{proposition}
See \cite {l3}, \cite{fkklr}, part V, Proposition IV.3.1. 
For further reference, note that $m_j$ is characterized by the relation
\begin{equation}\label{mj}
x\wedge x^{(2,y)}\wedge\dots \wedge \widehat{x^{(\rho-j+1,y)}}\wedge \dots \wedge (x^{(\rho+1,y)}= m_j(x,y) \beta_\rho(x,y)\ .\end{equation}
see (\cite {l3} section 16, or \cite{fkklr} Part V, section IV.3).

As a consequence, the right handside of \eqref{mmj} defines a polynomial on $\mathbb K \times V\times V$, still denoted by $m(T,x,y)$ and called the \emph{generic minimal polynomial} of $V$ over $\mathbb K$.

The following property is an obvious consequence of the construction of the minimal polynomial.
\begin{proposition} Let $g\in Str(V)$. Then, for all $x,y\in V$,
\[ m(T,gx,\sigma(g)y) = m(T,x,y)\ ,
\] 
where $\sigma (g) = {(g^t)}^{-1}$ if $\mathbb K = \mathbb R$ and $\sigma(g) = {(g^*)}^{-1}$ if $\mathbb K = \mathbb C$.
\end{proposition}

Let us consider more closely the case where $\mathbb K = \mathbb C$ and $\mathbb V$ is a PHJTS.  

\begin{proposition}\label{minpol}
 Let $\mathbb V$ be a PHJTS of rank $r$. 
 
 $i)$ The absolute rank $\rho$ is equal to $r$. 
 
 $ii)$  The generic minimal polynomial $m(T,x,y)$ (over $\mathbb C$) is holomorphic in $x$ and antiholomorphic in $y$.
 
$iii)$ Let  $(c_1,c_2,\dots, c_r)$ be a Jordan frame of $\mathbb V$, and let $x=\sum_{j=1}^r t_j c_j$. Then
\[ m(T,x,x) = \prod_{j=1}^r (T-t_j^2)\ ,\qquad m_k(x,x) = \sum_{i_1<i_2\dots <i_k} t_{i_1}^2\dots t_{i_k}^2\ .\]
\end{proposition}
See \cite{fkklr} Part V, proposition VI.2.6.

Define now the \emph{complex fundamental kernel} of $\mathbb V$ to be the polynomial $h$ on $\mathbb V\times \mathbb V$ defined by
\[h(x,y) = m(1,x,-y)\ .
\]
The polynomials $m_k$ are called the \emph{fundamental invariant polynomials} of $\mathbb V$. Notice the following relations

\begin{equation}\label{mh}
m(T,x,y) = T^rh(T^{-1} x,-y),\quad h(x,y) =\sum_0^r m_k(x,y),
\end{equation}
where $m_0\equiv 1$.

\begin{proposition} Let $\mathbb V$ be a PHJTS of rank $r$ and let $(c_1,c_2,\dots, c_r)$ be a Jordan frame of $\mathbb V$. The polynomial $h$ is uniquely determined by the following property :

$i)$ $h$ is holomorphic in $x$, antiholomorphic in $y$

$ii)$ $h(kx,ky) = h(x,y)$ for any $k\in Aut(\mathbb V)$ and $x,y\in \mathbb V$

$iii)$ for $x=\sum_{j=1} t_j c_j$, 
\[h(x,x) = \prod_{j=1}^r (1+t_j^2)\ .
\]

$iv)$ similar statements hold for $m_k$, $1\leq k\leq r$.

\end{proposition}
\begin{proof} Clearly $h$ satisfies $i)$ and $ii)$. As $m_j$ is homogeneous of degree $j$ in $y$,  $h$ satisfies $iii)$. Suppose now that $\widetilde h$ is a polynomial wihch satisfies the three properties. As $h$ and $\widetilde h$ are holomorphic in $x$ and antiholomorphic in $y$, it suffices to show that they coincide on the diagonal $\{(x,x), x\in \mathbb V\}$. Any orbit under $Aut(\mathbb V)$ in $\mathbb V$ meets the maximal flat $A =\oplus_{j=1}^r \mathbb R c_j$. As $h(x,x)$ and $\widetilde h(x,x)$ coincide on $A$, they coincide everywhere on $\mathbb V$.\end{proof}

For later reference, let us state the following result.

\begin{proposition} Let $(c_1,c_2,\dots, c_r)$ be a Jordan frame of $\mathbb V$. Let $x=\sum_{j=1} x_jc_j$ and $y=\sum_{j=1}^ry_jc_j$, where $x_j,y_j\in \mathbb C$. Then
\begin{equation}\label{hdouble}
h(x,y) = \prod_{j=1}^r (1+x_j\overline {y_j})\ .
\end{equation}
\end{proposition}

\begin{proof} Let $\mathbb A= \oplus_{j=1}^r \mathbb C c_j$ be the complexification of the maximal flat subspace $A= \oplus_{j=1}^r \mathbb  R c_j$. The restriction of $h$ to $\mathbb A \times \mathbb A $ is holomorphic in the first variable and antiholomorphic in the second. The same property is true for the right handside of \eqref{hdouble}. Hence it suffices to verify the equality for $x=y$. For $1\leq j\leq r$, let $e^{i\theta_j}$ be any complex number of modulus $1$. Then $(e^{i\theta_1}c_1, e^{i\theta_2} c_2, \dots, e^{i\theta_r} c_r)$ is a Jordan frame of $\mathbb V$. For all $t_1,t_2,\dots, t_r\in \mathbb R$,
\[ h( \sum_{j=1}^r t_je^{i\theta_j} c_j,  \sum_{j=1}^r t_je^{i\theta_j} c_j) = = \prod_{j=1}^r (1+t_j^2)= \prod_{j=1}^r \big(1+(t_je^{\theta_j})(\overline {t_je^{\theta_j}})\big)\ .
\]
As the $e^{i\theta_j}$ are arbitrary complex numbers of modulus $1$, the equality \eqref{hdouble} is true for $x=y\in \mathbb A$.
\end{proof}

Let $\mathbb V$ be a simple PHJTS. With respect to some Jordan frame $(c_1,c_2,\dots, c_r)$, the Peirce decomposition (over $\mathbb C$) reads
\[\mathbb V = \bigoplus_{j=1}^r \mathbb C c_j \bigoplus_{1\leq i<j\leq r} \mathbb V_{ij}\bigoplus_{j=1}^r \mathbb V_{0j} \ .
\]
Then define the characteristic numbers of $\mathbb V$  as
\[a_\mathbb C ={ \dim}_\mathbb C \mathbb V_{ij}, \qquad  b_\mathbb C={\dim}_\mathbb C \mathbb V_{0j}\ ,
\]
and let the complex genus of $\mathbb V$ be $g= a_\mathbb C(r-1) + b_\mathbb C+2$.
\begin{proposition}\label{powerg}
 Let $\mathbb V$ be a simple PHJTS, of rank $r$ and genus $g$. The complex canonical kernel $c(x,y)$ satisfies the identity
\[c(x,y) = h(x,y)^g\ .\]
\end{proposition}

For a proof, see \cite{fkklr} Part V, Proposition VI.3.6.

For a special class of PHJTS, we will need another interpretation of the generic minimal polynomial (equivalently of the complex fundamental kernel). A PHJTS $V$ is said to be of \emph{ tube type} or  of \emph{Jordan algebra type} is there exists a  tripotent $c$ such that $V=V_2(c)$. Observe that such a tripotent is necessary maximal and then the same property will then be valid for \emph{any} maximal  tripotent. If this is the case, then the product $x._c y = \{x,c,y\}$ endows $V$ with a structure of complex Jordan algebra, with unit element $c$. The map $Q(c)$ is a antiholomorphic involution of $V$ and its fixed points set is a Euclidean Jordan algebra\footnote{For the theory of Euclidean Jordan algebra, see \cite{fk}}. The converse is true. Let $J$ be a Euclidean Jordan algebra. Let $\mathbb J$ be its complexification, and define on $\mathbb J$ the trilinear product
\[ \{x,y,z\} = (x\overline y)z+(z\overline y)x -(xz)\overline y\ .\]
Then $\mathbb J$ is a PHJTS of tube type (called the \emph{Hermitification} of $J$).

Let $L(x)$ and $P(x)$ be the standard operators on $J$, and extend them to $\mathbb J$ as complex linear operators (see \cite{fk}). The operators $L(x,y)$ and $Q(x)$  of $\mathbb J$ are given by 
\[ L(x,y)= L(x\overline y) +L(x)L(\overline y) -L(\overline y)L(x), \quad Q(x) =2 L(x)^2-L(x^2) = P(x)\ .\]

The rank of the PHJTS $\mathbb J$ coincides with the rank of the Jordan algebra $J$, any Jordan frame of $J$ is a Jordan frame of $\mathbb J$. The structure group $Str(\mathbb J)$ is the complexification of $Str(J)$.
 
 Let $J$ be a simple Euclidean Jordan algebra. Then $\mathbb J$ is a simple PHJTS. Let $(c_1,c_2,\dots,c_r)$ be a Jordan frame of $J$. There is a corresponding Peirce decomposition $J=\sum_{j=1}^r \mathbb R c_j\bigoplus_{1\leq i<j\leq r} J_{ij}$.  The spaces $J_{ij}, i<j$ have all the same dimension, usually denoted by $d$. Then $\mathbb J = \sum_{j=1}^r  \mathbb C c_j\bigoplus_{1\leq i<j\leq r} \mathbb J_{ij}$ is the Peirce decomposition of $\mathbb J$. Hence the characteristic numbers of $\mathbb J$ as a simple PHJTS are $a_\mathbb C=d$ and $b_\mathbb C=0$.
 
 Let $\Delta$ be the determinant function of $J$. Also let $a_k$ be the polynomials on $J$ defined by
\[\Delta(Te-x) = T^r + \sum_{k=1}^r (-1)^k T^{\,r-k}\, a_k(x)\ .
\]
Let $(c_1,c_2,\dots, c_r)$ be a Jordan frame of $J$, and let $x=\sum_{j=1}^r t_jc_j$. Then
\[ \Delta(x) = \prod_{j=1}^r t_j\ , a_k(x) = \sum_{i_1<i_2<\dots<i_k} t_{i_1} t_{i_2} \dots t_{i_k}\ .\]
Extend $\Delta$ and the $a_k$'s as holomorphic polynomials on $\mathbb J$.

\begin{proposition}\label{PHJTSJA}
 Let $q(z,w)$ a polynomial on $\mathbb J\times \mathbb J$ which is holomorphic in $z$, antiholomorphic in $w$ and satisfies
\[ q(gz, \sigma(g) w) = q(z,w)\ ,\]
for any $g\in Str(\mathbb J)_0$. Then there exists a unique polynomial $p$
 on $J$, which is invariant by $Aut(J)_0$ such that, for any $x\in J$
 \[ q(x,x) =p(x^2)\ .\]
 For any $x\in J$,
 \[ q(x,e) = p(x)\ .\]
 Moreover, the correspondance $q\longmapsto p$ is $1$ to $1$.
 
 \end{proposition}
See \cite{fk}, Corollary XI.3.4 and Proposition XIV.1.1.

\begin{proposition} Let $h$ be the complex fundamental kernel of $\mathbb J$.
 For any $x\in \mathbb J$ and $1\leq j\leq r$,
\[ h(x,e) = \Delta(e+x), \quad m_k(x,e) = a_k(x)\ .\]
\end{proposition}
\begin{proof} The polynomial $h$ satisfies all assumptions of the previous proposition. Let $(c_1,c_2,\dots, c_r)$ be a Jordan frame of $J$, and let $x=\sum_{j=1}^r x_jc_j$. Then \[h(x,x) = \prod_{j=1}^r (1+x_j^2)= \Delta(e+x^2),\]
so that the associated polynomial $p$ on $J$ is equal to $\Delta(e+x)$. A similar argument holds for the relation between the polynomials $a_k$ and the $m_k$.
\end{proof}

\section{The fundamental kernel for a PJTS}

\subsection{ Hermitification of a PJTS }
Let us now examine the real theory, i.e. the case of a PJTS. The main idea is to use the fact that a PJTS is a real form of a PHJTS.

\begin{proposition}\label{minC}
 Let $V$ be a PJTS, and let $\mathbb V$ be its Hermitification. Then the restriction to $V\times V$ of the minimal polynomial (over $\mathbb C$) of $\mathbb V$ coincides with the minimal polynomial (over $\mathbb R$) of $V$.
\end{proposition}

\begin{proof} Let $\rho$ be the rank of $\mathbb V$. For $k\geq \rho+1$, $\beta_k$ (defined by \eqref{betak}) vanishes identically. The polynomial $\beta_\rho $  is  holomorphic in $x$ and antiholomorphic in $y$. Now $V\times V$ is a real form of the space $\mathbb V\times \mathbb V_{op}$. Hence $\beta_\rho(x,y)$ can not vanish identically on $V\times V$. Hence $\rho$ is the degree of the generic minimal polynomial of $V$. Now $m_j$ is characterized by the relation \eqref{mj}. When $(x,y)$ is a regular pair in $V\times V$,  $m_j(x,y)$ is real-valued and hence the restriction of $m_j$ to $V\times V$ is a real-valued polynomial, and the statement follows.
\end{proof}

As a corollary, the absolute rank of $V$ is equal to the rank of its Hermitification $\mathbb V$.

\begin{proposition} Let $V$ be a simple PJTS. Then its Hermitification $\mathbb V$ is simple (as a PHJTS), unless $V$ admits a complex structure $J$ for which $V$ is a PHJTS. If this happens, then $\mathbb V \simeq V\oplus V_{op}$, where $V_{op}$ is the PHJTS $V$ with the opposite complex structure given by $-J$. \end{proposition}

\begin{proof} Let $W$ be a non trivial (complex) ideal in $\mathbb V$, and consider $W\cap V$. It is clearly an ideal of $V$. Hence it is either $V$ or $\{0\}$. If $W\cap V = V$, then $V\subset W$, and hence $\mathbb V=V +iV=W$, a  contradiction. Hence $V\cap W =\{0\}$. 

Now let
\[ Y = \{ u\in V, \exists v\in V, u+iv\in W\}\ .\]
Then $Y$ is an ideal of $V$, and is not reduced to $\{ 0\}$, hence $Y=V$.
For any $v\in V$, there exists a unique element $v'$ in $V$ such that $v+iv'\in W$. This defines a  one-to-one map $J$ of $V$ into itself, such that $W = \{ u+iJu, u\in V\}$. The map $J$ is clearly $\mathbb R$-linear. 
For $u\in V$, $i(u+iJu)=-Ju+iu$ also belongs to $W$, so that $J^2u=-u$. As $u$ was arbitrary, $J^2=-\Id_V$. Now, for $u,v,w$ in $V$, 
\[\{ u+iJu,v,w\} = \{u,v,w\} +i\{Ju,v,w\}\]
belongs to $W$, and hence $J\{u,v,w\} = \{Ju,v,w\}$, showing that the triple product on $W$ is complex linear in the first variable (w.r.t. $J$), and, by symmetry, also in the third variable. By a similar argument,  $\{u,Jv,w\} = -J\{u,v,w\}$ for any triple $u,v,w\in V$. Hence $V$ admits a complex structure. 

The map $\iota : V\longmapsto W$ defined by $\iota (v) = v+iJv$ satisfies $\iota(Jv) = -i Jv$, so that $(W,i)$ is isomorphic to $(V,-J) = (V,J)_{op}$. Let  $W' = \{ v-iJv, v\in V\}$. Then $W'$ is an ideal in $\mathbb V$, $(W',i)$ is isomorphic to $(V,J)$ and $\mathbb V = W\oplus W'$.
\end{proof}

\subsection{ The fundamental kernel of a PHJTS}

\begin{proposition}
Let $\mathbb V$ be a PHJTS of rank $r$, and let $m(T,x,y)$ be the complex generic minimal polynomial of $\mathbb V$. When regarded as a real PJTS,  its absolute rank $\rho$ is equal to $2r$, and the (real) generic minimal polynomial  of $\mathbb V$ is given by

\[ m_\mathbb R(T,x,y) = m(T,x,y) \overline{m(T,x,y)} \ .\]
\end{proposition}

\begin{proof} Recall first that   be the complex generic minimal polynomial $m(T,x,y)$ is of degree $r$ in $T$ for any regular pair $(x,y)$. Let $(c_1,\dots,c_r)$ be a Jordan frame of $\mathbb V$, and let $c=c_1+\dots + c_r$ be the associated maximal tripotent. Let $x=\sum_{j=1}^r x_j c_j$, where the $x_j$ are complex numbers such that the $2r$ numbers $(x_j,\overline{x_k})1\leq j,k\leq r$ are all distinct. Consider, for $1\leq k\leq 2r$
$x^{(k,c)} = \sum_{j=1}^r x_j^k c_j$.  In the (real) basis $\{c_1,ic_1,c_2,\dots, c_r, ic_r,\dots, ic_r\}$ the coordinates of $x^{(k,c)}$ are 

\[\big(\Re (x_1^k), \Im (x_1^k),\dots, \Re (x_r^k), \Im (x_r^k)\big)\ .\]

By a Van der Monde type argument, the determinant of the $2r$ coordinates of this family of $2r$ vectors is easily seen to be different from $0$. So the absolute rank of $\mathbb V$ (over $\mathbb R$) is at least $2r$. Now for $(x,y)$ regular, $\mu(T,x,y) =m(T,x,y)\overline {m(T,x,y)}$ is a monic polynomial of degree $2r$ in $T$ with real coefficients, such $T\mu(T,x,y)$ takes the value $0$ when substituting $T=x$. This shows that the generic minimal polynomial of $\mathbb V$ (over $\mathbb R$) is of degree $2r$, and is equal to the product of the (complex) generic minimal polynomial with its conjugate.
\end{proof}

Define the (real) \emph{fundamental kernel} of the PHJTS $V$ to be 
\[ k(x,y) = h(x,y)\overline {h(x,y)}\ .\]

\begin{proposition} Let $V$ be a simple PHJTS, ansd let  $k(x,y)$be its fundamental kernel. Then
\[c(x,y) = k(x,y)^{\frac{p}{2}}\]
 where $p$ is the genus (over $\mathbb R$) of $V$.
\end{proposition}

\begin{proof} Let $C(x,y)$ be the complex dual Bergman operator of $V$. Recall that the complex canonical kernel is $c_\mathbb C(x,y) ={ \det}_\mathbb C C(x,y)$. The canonical kernel (over $\mathbb R$) is $c(x,y) = {\det}_{\mathbb R} C(x,y)= \det _\mathbb C(x,y) \,\overline{\det _\mathbb C(x,y) }$. Hence
\[ c(x,y) = c_\mathbb C  (x,y) \overline{c_\mathbb C (x,y)}= h(x,y)^g \overline {h(x,y)}^g= k(x,y)^g
\]

Now, if $V=\bigoplus_{j=1}^r \mathbb C c_j \bigoplus_{1\leq i<j\leq r} V_{ij}\bigoplus_{1\leq j\leq r} 
V_{j0}$ is  a Peirce decomposition of $V$, all subspaces are complex subspaces, and as such, the various dimensions over $\mathbb R$ are twice their dimensions over $\mathbb C$. Hence $p= 2g$. The statement follows.
\end{proof}

\subsection{ The fundamental kernel for a reduced PJTS}

Consider a simple PJTS $V$ and let $\mathbb V$ be its Hermitification. The PJTS $V$ is said to be \emph{reduced} if a primitive tripotent of $V$ is primitive in $\mathbb V$. As $\mathbb V(c,1) = \mathbb C c$ for a primitive tripotent of $\mathbb V$, this is equivalent to the condition that $V(c,1) = \mathbb R c$ for a primitive tripotent of $V$. Notice that if true for one primitive tripotent, then it is true for any primitive tripotent. It also amounts to the condition that the characteristic number $c$  of $V$ is equal to $1$.

If $V$ admits a complex structure (that is if $V$ is a PHJTS), then $V$ is \emph{not} reduced. A PJTS $V$  will be said to be \emph{non-reduced} if it has no  complex structure and is not reduced. Simple PJTS fall into three families : those admitting a complex structure (=PHJTS), the reduced PJTS and the non-reduced PJTS.

\begin{proposition} Let $V$ be a  simple reduced PJTS of rank $r$.  Then its Hermitification $\mathbb V$ is a simple PHJTS of rank $r$ and the absolute rank of $V$ is equal to $r$.
\end{proposition}

\begin{proof} Let $(c_1,c_2,\dots, c_r)$ be a Jordan frame of $V$. Let $c=c_1+\dots +c_r$. The kernel of $L(c,c)$ in $\mathbb V$ is the complexification of its kernel in $V$, hence is $\{0\}$ as $c$ is a maximal tripotent of $V$. Hence $c$ is a maximal tripotent of $\mathbb V$. As the $c_j$ are, by assumption primitive tripotents (and still orthogonal) of $\mathbb W$,  $(c_1,c_2,\dots, c_r)$ is a Jordan frame of $\mathbb V$, so the rank of $\mathbb V$ is $r$, and hence the absolute rank of $V$ is $r$ (cf  Proposition \ref{minC}).

\end{proof}

Let $h(x,y)$ be the complex fundamental kernel of $\mathbb V$. Then for $x,y\in V$, $h(x,y)$ is real (a consequence of Proposition \ref{minC}). Define the \emph{fundamental kernel} of $V$ by \[k(x,y) = h(x,y)^2\ .\]

\begin{proposition} Let $V$ be reduced simple PJTS. Then
\[ c(x,y) = k(x,y) ^{\frac{p}{2}}\ .\]
\end{proposition}
\begin{proof} Let $C_\mathbb V$ (resp. $C_V$) be the Bergman kernel of $\mathbb V$ (resp. $V$). For $x,y$ in $V\times V$, $C_\mathbb V(x,y)$ is  the complexification of $C_V(x,y)$ of $V$. Hence the restriction to $V\times V$ of the complex canonical kernel $c_\mathbb V$ of $\mathbb V$ coincides with the canonical kernel of $V$ (this argument is valid for any PJTS). Hence 
\[ c(x,y) = c_\mathbb V(x,y) = h(x,y)^g = k(x,y)^{\frac{g}{2}}\ .\] 
Now a Jordan frame of $V$ is a Jordan frame of $\mathbb V$, and the corresponding Peirce decomposition of $\mathbb V$ is just the complexification of the Peirce decomposition of $V$. So the characteristic numbers (over $\mathbb  R$) for $V$ are the same as characteristic numbers (over $\mathbb C)$ for $\mathbb V$. Hence $p=g$, and the statement follows.
\end{proof}

\subsection{The fundamental kernel for a non reduced PJTS}

For the non-reduced case, we first need a lemma (implicitly contained in \cite{l4} section 11, see also \cite{d}). Let $n\geq 2$ be an integer, and let $V^{(n)}= \mathbb   R^n$ with the triple product given by
\[\{x,y,z\} = (x,y)z+(z,y)x-(x,z)y\ ,
\]
where $(x,y)$ is the standard Euclidean product on $\mathbb R^n$. Then $V^{(n)}$ is a PJTS.
Let $c=e_1$ be the first vector of the standard basis of $\mathbb R^n$. Then
$L(c,c) = \Id$ and for  $x=(x_1,x_2,\dots, x_n)$, 
\[ Q(c)x=(x_1,-x_2,\dots,- x_n)\ .\]
Hence $V^{(n)}(c)^+ = \mathbb Rc$ and $V^{(n)}(c)^-=\{x_1=0\}$.
The Jordan algebra product is given by  
\[x._cz = (x_1z_1-x_2z_2-\dots -x_nz_n,\,x_1z_2+z_1x_2,\dots, x_1z_n+z_1x_n)\ ,
\]
for $x=(x_1,x_2,\dots,x_n)$ and $z=(z_1,z_2,\dots,z_n)$.

The Hermitification of $V^{(n)}$ is  $\mathbb V^{(n)}= \mathbb C^n$ with the triple product given by
\[\{x,y,z\} = (x, \overline y) z+ (z,\overline y) x -(x,z)\overline y\ ,
\]
where the the inner product on $\mathbb R^n$ is extended as a complex symmetric bilinear form.

The element $c=e_1$ is still a maximal tripotent, which is no longer primitive in $\mathbb V^{(n)}$. In fact, let 
\[d = \frac{1}{2}(e_1+ie_2),\quad \overline d = \frac{1}{2}(e_1-ie_2)\ .\]
Then $d$ and $\overline d$ are orthogonal primitive tripotents in $\mathbb V^{(n)}$, and $c=d+\overline d$, so that $(d,\overline d)$ is a Jordan frame for $\mathbb V^{(n)}$. 

Hence $V^{(n)}$ is a PJTS of rank $1$, of Jordan algebra type and which is not reduced.

\begin{lemma} Let $V$ be a PJTS of rank $1$, which is of Jordan algebra type and not reduced. Then $V$ is isomorphic to $V^{(n)}$ for $n=\dim V$.
\end{lemma}

\begin{proof} By assumption, there exists a tripotent $c$, such that $L(c,c)=\Id$. As $c$ is primitive, $V_+ = \mathbb Rc \oplus V_-$, where $V_- = \{x,Q(c)x=-x\}$. Let $x\,._c\,y=\{x,c,y\}$ be the corresponding Jordan algebra structure on $V$.
Let $u,v$ in $V_-$. Then, as $Q(c)$ is an isomorphism of the Jordan algebra structure,
\[ Q(c)(u\,._c\,v) = Q(c)u\,._c\,Q(c)v=(-u)\,._c\,(-v)= u\,._c\,v\ ,\]
so that $u._cv$ belongs to $V_+$, hence  $u\,._c\, v = \alpha(u,v)c$, where $\alpha$ is a symmetric bilinear form on $V_-$. Moreover, $Q(c)$ is a Cartan involution of the Jordan algebra $(V, ._c)$, so that $\Tr (L_c(u\,._c\,v))$ is negative-definite on $V_-\times V_-$. As \[\Tr L_c(u\,._c\,v)= \Tr (\alpha(u,v) L_c(c)) =\alpha(u,v)\, n\ ,\]the form $\alpha$ is negative-definite. Choose an orthonormal basis $(e_2,\dots, e_n)$ of $V_-$ for the form $-\alpha$, and let $e_1=c$. Then the Jordan product on $V$ can be written as follows : for
$x=(x_1,x_2,\dots, x_n), y = (y_1,y_2,\dots, y_n)$,
\[x._c y = (x_1+y_1-x_2y_2-\dots -x_n y_n, x_1y_2+y_1x_2,\dots, x_1y_n+y_1x_n\ .\]
So the Jordan algebra $(V, ._c)$ is isomorphic to the Jordan algebra $V^{(n)}(c)$. Now, 
\[ Q(x) = P_c(x) \circ Q(c), \]
and $\{x,y,z\} = \big(\frac{1}{2}Q(x+z)-Q(x)-Q(z)\big) y$,
so that $V$ and $V^{(n)}$ are isomorphic as JTS.
\end{proof}

\begin{proposition}\label{NRJF}
 Let $V$ be a simple non-reduced PJTS of rank $r$. Then its Hermitification $\mathbb V$ is a simple PHJTS of rank $2r$. Moreover, given a Jordan frame $c_1,c_2,\dots, c_r$ of $V$, there exist primitive orthogonal tripotents $d_1,d_2,\dots, d_r$ in $\mathbb V$ such that $c_j=d_j+\overline d_j$ and $(d_1,\overline d_1,\dots, d_r,\overline d_r)$ is a Jordan frame of $\mathbb V$.
\end{proposition}
\begin{proof} Let $(c_1,c_2,\dots, c_r)$ be a Jordan frame in $V$. Then for any $j,1\leq j\leq r$ $V_2(c_j)$ is a PJTS of Jordan algebra type and rank $1$ (as $c_j$ is primitive) and is not reduced. By the previous Lemma, there exist $d_j, \overline {d_j}$ two orthogonal primitive tripotents in $\mathbb V_2(c_j)$ such that $c_j=d_j+\overline d_j$. Hence $(d_1,\overline{d_1},\dots, d_r,\overline{d_r})$ is a Jordan frame of $\mathbb V$ and the statement follows.

\end{proof}

For $1\leq k\leq r$ let  $\overline k = k+r$ and $d_{\overline k} = \overline{d_k}$. The Pierce decomposition of $V$ w.r.t. $(c_1,c_2,\dots, c_r)$ and the Pierce decomposition of $\mathbb V$ w.r.t. $(d_1, d_{\overline 1}, \dots, d_r, d_{\overline r})$ are connected by the following relations :
\medskip

$\bullet \quad V_{ii}^\mathbb C = \mathbb C\, d_i \oplus \mathbb C\, \overline{d_i}\oplus \mathbb V_{i\overline i}$

\medskip

$\bullet \quad V_{ij}^\mathbb C = \mathbb V_{ij}\oplus \mathbb V_{i \overline j} \oplus \mathbb V_{\overline i j} \oplus V_{\overline i \,\overline j}$

\medskip

$\bullet \quad V_{i0}^\mathbb C = \mathbb V_{io} \oplus \mathbb V_{\overline i 0}\ $
\medskip

\noindent
for $1\leq i<j\leq r$. The corresponding dimensions (= characteristic numbers) are related by
\begin{equation}  b=2\,b_\mathbb C,\quad c= 2+a_\mathbb C\ .
\end{equation}
and $a=0$ if $r=1$, $a=4 a_\mathbb C$ if $r\leq 2$.
Further, the genus $p$ is equal to $2\,g$, as
\[p= 2c+(r-1)a+b = 2(2+a_\mathbb C) + (r-1) 4a_\mathbb C +2 b_\mathbb C = 2\big( 2 +(2r-1)a_\mathbb C + b_\mathbb C) = 2 g\ .
\]
For the proof of the next proposition, the following elementary lemma will be useful.

\begin{lemma}\label{orb}
 Let $V$ be a PJTS. Any element $x$ of $V$ is conjugate under $Str(V)_0$ to a tripotent. If all eigenvalues of $x$ are strictly positive, then $x$ is conjugate under $Str(V)_0$ to a maximal tripotent.

\end{lemma}

\begin{proof} Let $(c_1,c_2,\dots, c_r)$ be a Jordan frame of $V$ and let $x=\sum_{j=1}^r t_j c_j$. Let $J=\{j\in \{ 1,2,\dots, r\}, t_j\neq 0\}$. For $j\in J$, $T_j= L(c_j,c_j)$ belongs to $\mathfrak{str}(V)$, $T_j c_j = c_j$, and for $i\neq j$, $ T_j c_i = 0$. Moreover, the operators $T_j$ mutually commute. Hence 
\[ \exp \big(-\sum_{j\in J} \log (t_j)\, T_j\big)(\sum_{j\in J} t_j c_j) = \sum_{j\in J} c_j\ .
\]
The first statement follows. The second statement corresponds to the case where $J = \{1,2,\dots,r\}$.
\end{proof}
\noindent
{\bf Remark} The orbits of $Str(V)_0$ in a simple PJTS $V$ are described in \cite{fkklr} Part II, Theorem II.2.5.

\begin{proposition}\label{posNR}
 Let $V$ be a simple non-reduced PJTS. Then the complex fundamental kernel $h$ of its Hermitification $\mathbb V$ satisfies $h(x,y)\geq 0$ for all $x,y\in V\times V$.
\end{proposition}
\begin{proof} Let $(c_1,c_2,\dots, c_r)$ be a Jordan frame of $V$. Let $x=\sum_{j=1} x_jc_j$ and $y=\sum_{j=1} y_j c_j$, where $x_j,y_j\in \mathbb R$. Let $(d_1,\overline {d_1}, \dots, d_r,\overline {d_r})$ be an associated Jordan frame of $\mathbb V$ as in Proposition \ref{NRJF}. Then $x= \sum_{j=1}^r x_jd_j + x_j \overline {d_j}$ and $y=  \sum_{j=1}^r y_jd_j + y_j \overline {d_j}$. By Proposition \ref{minC} and \eqref{hdouble}, 
\[h(x,y) = \prod_{j=1}^r (1+x_jy_j)^2,\]
hence $h(x,y)\geq 0$. In other words, the property to be proven is true if $x$ and $y$ belong to a common maximal flat subspace of $V$. Let $c$ be a maximal tripotent in $V$. For any $x$ in $V_2(c)$, there exists a Jordan frame $(c_1,c_2,\dots, c_r)$ such that $c=c_1+c_2+\dots c_r$ and $x= \sum_{j=1}^r x_j c_j$, so $x$ and $c$ belong to a common maximal flat. Hence, $h(c,x)\geq 0$. Let now $x$ be in $V$, and let $x=x_2+x_1$ where $x_2\in V_2(c)$ and $x_1\in V_1(c)$. Then, by \eqref{V2} $h(c,x_2+x_1) = h(c,x_2)\geq 0$. Now for any $g\in Str(V)$, any $x\in V$ and any maximal tripotent $c$
\[ h(x,gc)= h(\sigma(g)x,c)\geq 0. \]
Thus, if $y$ has all its eigenvalues strictly positive, Lemma \ref{orb} shows that $h(x,y)\geq 0$. As the set of elements with all eigenvalues strictly positive is an open dense set in $V$, the result follows by continuity.
\end{proof}

\section{Bernstein-Sato identity for the complex fundamental kernel of a PHJTS of tube type}

Let  $J$ be a simple Euclidean Jordan algebra, with neutral element $e$. Let  $D=D(x,\frac{\partial}{\partial x})$ be a linear differential operator  on some open subset $\mathcal O$ with $\mathcal C^\infty$ coefficients. Define its \emph{symbol} as the function $\sigma_D$ on $\mathcal O\times J$ defined by
\[ D\, e^{(x,\,\xi)} = \sigma_D(x,\xi)\, e^{(x,\,\xi)}\ .\]
The symbol is a polynomial in $\xi$, whose coefficients are $\mathcal C^\infty$ functions of $x$.

Let $g$ be a linear transformation of $J$, and denote by $L(g)$ the left action on $\mathcal C^\infty(J)$ defined by $L(g)f = f\circ g^{-1}$. A differential operator $D$ is invariant under $g$ (i.e. commutes with $L(g)$) if and only if $\sigma_D (gx,\sigma(g)\xi) = \sigma_D(x,\xi)$,
where $\sigma(g) = ({g^t})^{-1}$.

Let $\Omega$ be the \emph{open cone of squares} of $J$ (see \cite{fk} for details), and denote by $H$ the connected component of the group $G(\Omega)$ (equivalently of $Str(J)$). By Proposition \ref{PHJTSJA}, to each $Aut(J)_0$-invariant polynomial $p$ on $J$, one associates the $H$-invariant differential operator $q(x,\frac{\partial}{\partial x})$. We may apply this to the  coefficients $a_1,\dots, a_r$ of the generic minimal polynomial.

Let $M_0=\Id$, and for $1\leq k\leq r$, denote by $M_k$ the $H$-invariant differential operator  on $J$ such that 
\[ \sigma_{M_k} (e,\xi) = a_k(\xi)\ .\]
(cf \cite{fk} ch. XIV). Although we won't use the result, notice that the $M_k$ generate the algebra of $H$-invariant differential operators on $J$.

Let $\Omega$ be the \emph{open cone of squares} in $J$ (see \cite{fk}). Recall that $\Delta(x)>0$ for $x\in \Omega$.

A large family of $H$-invariant differential operators on $J$ is provided by the following result (see \cite{fk} Proposition XIV.1.5).

\begin{proposition} let $s$ be a complex number, and let $D_s$ be the differential operator defined (a priori on $\Omega$) by
\[ D_s = \Delta(x)^{1+s}\circ \Delta\big(\frac{\partial}{\partial x}\big)\circ \Delta(x)^{-s}\ .\]
Then 

$i)$ $D_s$ extends to a differential operator on $J$ with polynomial coefficients

$ii)$ $D_s$ is $H$-invariant

$iii)$ 
\begin{equation}
 D_s= \sum_{k=0}^r (-1)^k  \prod_{j=1}^k \big(s-(j-1){\frac{d}{2}}\big) M_{r-k}\ .
 \end{equation}

\end{proposition}
See \cite{fk} ch. XIV.

Let $\iota$ be the inversion on $\Omega$ defined by $\iota(x) = x^{-1}$. Whenever $D$ is a $H$-invariant operator on $\Omega$, then the differential operator $D^{\,\iota}$ defined by
\[D^{\,\iota} f = \big(D(f\circ \iota)\big)\circ \iota
\]
can be shown to be also $H$-invariant. In particular, we have the following result.

\begin{equation}\label{dinv}
 D_s^{\,\iota} = (-1)^r D_t\ ,
 \end{equation}
where $t = \frac{d}{2}(r-1) -s$. See \cite{fk} proposition XIV.1.8.

Recall the \emph{Bernstein-Sato identity} for the determinant polynomial of the Euclidean Jordan algebra $J$.

\begin{proposition}  Let $s$ be a complex number. For $x$  in $\Omega$\footnote{ As $\Delta(x)>0$ for $x\in \Omega$, $\Delta(x)^s$ is defined without ambiguity.}
 \begin{equation}\label{bsomega}
\Delta(\frac{\partial }{\partial x}) \Delta^s (x) = b(s) \Delta(x)^{s-1}\ ,
\end{equation}
where
\begin{equation}\label{b}
b(s) = b_{r,d} (s) = s(s+\frac{d}{2})\dots(s+(r-1)\frac{d}{2})\ .
\end{equation}
\end{proposition}

The result is proved in \cite{fk} Prop. VII.1.4. 

\begin{proposition} For $s$ in $\mathbb C$, let $s^* = s+\frac{d}{2}(r-1)$. Then, for $x$ in $\Omega$
\begin{equation} \label{BS(e+x)}
(-1)^r\,\Delta_{s^*} \big(\Delta(e+x)^s\big) = b(s) \Delta (e+x)^{s-1}\ .
\end{equation}
\end{proposition}
\begin{proof}
Using the invariance of $\Delta\big(\frac{\partial}{\partial x}\big)$ under translations, for $x\in \Omega$,
\[\Delta\big(\frac{\partial}{\partial x}\big) \Delta(e+x)^s= b(s) \Delta(x+e)^{s-1}\ .\]
 Recall the identity, valid for $x$ invertible (in particular for $x$ in $\Omega$) :

\[\Delta(e+x) = \Delta (x) \Delta(e+x^{-1}) \ .
\]
Hence
\[\Delta(x)^{-s+1}\Delta(\frac{\partial}{\partial x})\, \big(\Delta(x)^s \Delta(e+x^{-1})^s\big) = b(s) \Delta(e+x)^{s-1}\Delta(x)^{-s+1}\ .
\] 
Setting for a while $d_s(x)= \Delta(e+x)^{s} $, the last identity can be rewritten as

\[ D_{-s} (d_s\circ \iota)(x) = b(s)\, d_{s-1}(x^{-1})
\]

Let $s^* = s+\frac{d}{2}(r-1)$. Then, using \eqref{dinv}, the last identity can be rewritten as
\[(-1)^r  D_{s^*}\, \big(d_s(x) \big)=b(s)\, d_{s-1}(x)\ ,
\]
which, up to notation agrees with \eqref{BS(e+x)}. 
\end{proof}

The advantage of this new Bernstein-Sato identity is that now the differential operator involved is invariant under $H$. For sake of simplicity, let $E_s$ be the differential operator with polynomials coefficients on $V$ defined by
\begin{equation}\label{EsMk}
E_s = (-1)^r D_{s^*} =  \sum_{k=0}^r (-1)^{r-k} \big(\prod_{j=1}^k(s+\frac{d}{2}(r-j)\big) M_{r-k}\  .
\end{equation}

Recall the connection between $\Delta$ and the complex canonical kernel $h$ of the Hermitification $\mathbb J$ of $J$ (see proposition \ref{PHJTSJA}).

\begin{lemma} Let $x,y$ in $\Omega$. Then
\[h(x,y)>0\ .
\]
\begin{proof} As $h(e,y) = \Delta(e+y)$ for any $y\in J$, $h(e,y)>0$ for any $y\in \Omega$. Let $g$ be in $H$. Then by invariance of $h$ under the action of $H$, $h(ge, {(g^t)}^{-1} y)>0$ for any $y$ in $\Omega$, hence $h(ge,y)>0$ for any $y\in \Omega$. As $g$ runs through $H$, $ge$ runs through all of $\Omega$. Hence the positivity of $h$ on $\Omega\times \Omega$.
\end{proof}
\end{lemma}

\begin{proposition} For $x,y$ in $\Omega$,
\begin{equation}\label{BSEJA}  E_s(x, \frac{\partial}{\partial x})\,h(x,y)^s = b(s) h(x,y)^{s-1}\ .
\end{equation}

\end{proposition}

\begin{proof}
Let $y$ be in $\Omega$. As $H$ is transitive on $\Omega$, choose $g\in H$ such that $\sigma(g)e=y$. Then $h(x,y)=h(x,\sigma(g)e) = h(g^{-1} x,e)$. 
Further, by the invariance of $E_s$ under the action of $H$
\[ E_s\, h(x,y)^s =E_s\, h(g^{-1} x,e)^s\]
\[=\big(E_s\, h(.,e)^s\big)(g^{-1}x)=
b(s)\,h(g^{-1} x,e)^{s-1}= b(s)h(x,y)^{s-1}\ .
\]
\end{proof}

We now extend this Bernstein-Sato identity to $\mathbb J$. 

If $D$ is any differential operator on a real vector space $E$ with polynomial coefficients, we extend it as a holomorphic differential operator $\mathbb D$ with holomorphic polynomial coefficients  on the complexification $\mathbb E\otimes_\mathbb R \mathbb C$ as follows : choose coordinates on $E$, and to to $D=x^\beta \frac{\partial^{\vert \alpha\vert}}{\partial x^\alpha}$, where $\alpha$ and $\beta$ are multiindices, associate $\mathbb D =z^\beta\frac{\partial^{\vert \alpha\vert} }{\partial z^\alpha}$, and extend this correspondance linearly. The extension does not depend on the choice of the coordinates on $E$. It satisfies the \emph{restriction principle} : for a holomorphic function $f$ defined on a neighborhood $\mathcal O$ of some point $x$ in $E$,
 \begin{equation}\label{complexD}
 \mathbb D f (x) = D(f_{\vert \mathcal O\cap E})(x)\ 
 \end{equation}
 The symbol $\sigma_{\mathbb D}$ of a holomorphic differential operator $\mathbb D$ is defined by
 \[\mathbb D(z,\frac{\partial}{\partial z}) e^{(z,\zeta)} =\sigma_\mathbb D (z,\zeta) e^{(z,\zeta)}\ ,
 \]
 where the inner product on $J$ is extended to a Hermitian form on $\mathbb J$.
The symbol is a antiholomorphic polynomial is $\xi$. If $\mathbb D$ is the holomorphic extension of a differential operator $D$ on $E$, then $\sigma_\mathbb D$ is the extension of $\sigma_D$ which is holomorphic in $z$ and antiholomorphic in $ \zeta$.

Apply the procedure of holomorphic extension to  the operators $E_s$ and $M_k$ to obtain a holomorphic differential operators $\mathbb E_s$ and $\mathbb M_k$ on $\mathbb J$.

\begin{theorem} Let $z,w$ be in $\mathbb J\times \mathbb J$, and assume that $h(z,w)\neq 0$. The following identity holds :

\begin{equation}\label{BSHJA}
 \mathbb E_s(z, \frac{\partial}{\partial z})\, h(z,w)^s = b(s) h(z,w)^{s-1}\ ,
\end{equation}
where $h(z,w)^s$ and $h(z,w)^{s-1}$ are computed from the same local determination of $\log h(z,w)$.
\end{theorem}

\begin{proof}
Let first assume that $w=y$ is in $\Omega$. Let
\[ \mathcal O_y= \{ z\in \mathbb J, h(z,y)\neq 0\}\ .\]
As $(z,y)$ is a holomorphic polynomial in $z$, $\mathcal O_y$ is a pathwise connected open subset of $\mathbb J$, which contains $\Omega$. For any $z_0$ in $\mathcal O_y$, one can select a path $\Gamma$ from $e$ to $z_0$ which is entirely contained in $\mathcal O_y$. By analytic continuation along $\Gamma$, there is a unique  determination of $\log h(z,y)$ in a neighborhood of $\Gamma$ which, for $z=x$ in a neighborhood of $e$, is equal to $\log h(x,y)$, with corresponding determinations of $h(z,y)^s$ and $h(z,y)^{s-1}$. The restriction principle \eqref{complexD} allows to compute the left handside for $z$ in $\Omega$ (close to $e$) by using \eqref{BSEJA}, and it agrees with the right handside. In a neigborhood of the path,  the two sides of \eqref{BSHJA} are holomorphic functions. Hence they coincide everywhere along the path, in particular at $z_0$. This proves the result, for the chosen determination of $\log h(z,y)$. But another determination differ from the one chosen by $2ki\pi$ for some integer $k$, hence changes the determination of $h(z,y)^s$  by a factor $e^{2ki\pi s}$. As
 $e^{2ki\pi(s-1)}= e^{2ki\pi s}$, the change is the same for both sides of the identity. Hence, the result obtained so far is true for \emph{any} determination of $\log h(z,y)$. To fully state \eqref{BSHJA}, it remains to get rid of the condition $w=y\in \Omega$. Now observe that both sides are antiholomorphic functions of $w$, when properly defined. Details (which are similar to the arguments we gave for the holomorphic extension in the $z$ variable) are left to the reader.
\end{proof}

The identity \eqref{BSHJA} is a \emph{Bernstein-Sato} identity for the complex fundamental kernel of $\mathbb J$, that is to say for a PHJTS of tube-type. 
 
 \section{Bernstein-Sato identity for the complex fundamental kernel of a PHJTS}

We now examine the case of a general simple PHJTS $\mathbb V$. 

Let $c$ be a maximal tripotent of $\mathbb V$. Let $\mathbb V=\mathbb V_2\oplus \mathbb V_1$ be the corresponding Peirce decomposition of $\mathbb V$. The space $\mathbb V_2$ is a sub-PHJTS which is of tube-type. It has a natural structure of complex Jordan algebra by setting
\[x._cy = \{x,c,y\}
\]
The map $Q(c)$ preserves $\mathbb V_2$ and induces a Cartan involution of $\mathbb V_2$. More precisely, let
\[J=\{ x\in \mathbb V_2, Q(c)x=x\}\ .
\]
Then $J$ is a real form of $\mathbb V_2$ and can be shown to be a Euclidean Jordan algebra.

\begin{proposition}Let $\mathbb V$ be a simple PHJTS. Let $c$ be a maximal tripotent and let $\mathbb V=\mathbb V_2\oplus \mathbb V_1$ be the corresponding Peirce decomposition. Then $\mathbb V_2$ is a simple PHJTS  of tube-type.
\end{proposition} 

\begin{proof} The only point to be checked is the simplicity of $\mathbb V_2$. If $\mathbb V$ is of rank $1$, there is nothing to prove. So we may assume that $\mathbb V$ has rank $\geq 2$. Suppose $\mathbb V_2$ would split as $\mathbb V_2 = Y\oplus Z$, a sum of two orthogonal ideals of $\mathbb V_2$. Let decompose $c=y+z$ with $y\in Y$ and 
$z\in Z$. From $\{y+z,y+z,y+z\} = y+z$ and the fact that both $Y$ and $Z$ are ideals of $\mathbb V_2$, one deduces that $y$ and $z$ are orthogonal tripotents. If $y=0$, then $c\in Z$, and hence $L(c,c)t=0$ for any $t$ in $Y$, yielding a contradiction. Hence both $y$ and $z$ are different from $0$. Let further decompose $y$ (resp. $z$) as a sum of primitive tripotents of $Y$ (resp. of $Z$). A primitive tripotent of $Y$ is a primitive tripotent of $\mathbb V_2$, and so we would have found two primitive tripotents $d$ in $Y$ and $f$ in $ Z$, such that the space
$\mathbb V(d,\frac{1}{2})\cap \mathbb V(f,\frac{1}{2}) \subset Y\cap Z= \{ 0\}$. But $d$ and $f$ are orthogonal primitive tripotents in $V$, and hence $\dim \big(\mathbb V(d,1)\cap \mathbb V(f,1)\big) = a_\mathbb C\neq 0$ as the rank of $\mathbb V_2$ is at least $2$, thus yielding a contradiction. 
\end{proof}

Let $m(T,x,y)$ be the generic minimal polynomial of $\mathbb V$, and let $m^{(2)}(T,x,y)$ be the generic minimal polynomial of $\mathbb V_2$. 

\begin{proposition}\label{V22}
For $x,y\in \mathbb V_2$,
\[ m(T,x,y) = m^{(2)}(T,x,y)\ .\]
\end{proposition}
\begin{proof} The two PHJTS $\mathbb V$ and $\mathbb V_2$ have the same rank, and hence a regular pair $(x,y)$ in $\mathbb  V_2$ is a regular pair in $\mathbb V$. So, the minimal polynomial $m^{(2)}(T,x,y)$ in $\mathbb V_2$ is equal to the minimal polynomial $m(T,x,y)$ in $\mathbb V$. By density, the result can be extended to any pair in $\mathbb V_2$. 
\end{proof}

Let $h$ be the complex fundamental kernel of $\mathbb V$, and $h^{(2)}$ the complex fundamental kernel of $\mathbb V_2$. Similarly, for $1\leq k\leq r$, let $m_k$ and $m_k^{(2)}$ be the fundamental invariant polynomials of $\mathbb V$ and of $\mathbb V_2$.

\begin{proposition}\label{restrick}
 Let $y_2\in \mathbb V_2$, and let $x=x_2+x_1$ be in $\mathbb V$. Then the following identities holds :
\medskip

$i)$ $h(x,y_2)= h(x_2,y_2)=h^{(2)}(x_2,y_2)$
\medskip

$ii)$ $m_k(x,y_2)= m_k(x_2,y_2)= m_k^{(2)} (x_2,y_2),\quad 1\leq k\leq r$
\end{proposition}

\begin{proof} 
For $i)$, the first equality is a consequence of \eqref{cV2}, combined with \eqref{powerg}.  The second equality follows from Proposition \ref{V22} and \eqref{mh}.
\end{proof}

As $h(y,x) = \overline {h(x,y)}$ for all $x,y\in \mathbb V$, symmetric results (i.e. exchanging the role of $x$ and $y$) are also true.

For $0\leq k\leq r$, let $M_k$ be the holomorphic differential operator with polynomial coefficients on $V$ associated to $m_k$, and similarly, let $ M^{(2)}_k$ be the differential operator on $V_2$ associated to the symbols $m_k^{(2)}$. 

\begin{proposition}\label{restrick2}
 Then, for any function $f$ in $\mathcal C^\infty(\mathbb V)$ and $x_2\in \mathbb V_2$,
\[  M_k f (x_2) =  M_k^{(2)}\big( f_{\vert \mathbb V_2}\big)(x_2)\ .\]
\end{proposition}

\begin{proof} Let $\xi= \xi_2+\xi_1$ be in $\mathbb V_2$. The symbol $m_k(x_2,\xi)$ at $x_2$ does depend only on $\xi_2$, and is equal to $m_k^{(2)}(x_2,\xi_2)$. Hence the action of the differential operator $M_k$ at $x_2$ use only differentiations in the directions of $\mathbb V_2$. The statement follows. See similar results in \cite{fkklr}, part III, section VI.4.4.
\end{proof}

With \eqref{EsMk} in mind, set the following definition :

\begin{equation}
\mathbb E_s =  \sum_{k=0}^r (-1)^{r-k} \big(\prod_{j=1}^k(s+\frac{d}{2}(r-j)\big) M_{r-k}\ .
\end{equation}

\begin{theorem} Let $z,w$ be in $\mathbb V\times \mathbb V$, and assume that $h(z,w)\neq 0$. The following identity holds :

\begin{equation}
\mathbb E_s(z,\frac{ \partial}{\partial z})  h(z,w)^s = b(s) h(z,w)^{s-1}\ ,
\end{equation}
where $h(z,w)^s$ and $h(z,w)^{s-1}$ are computed from the same local determination of $\log h(z,w)$.
\end{theorem}

\begin{proof} Let $c$ be a maximal tripotent of $\mathbb V$, and let $\mathbb V=\mathbb V_2\oplus \mathbb V_1$ be the associated decomposition of $\mathbb V$. Without loosing generality, we may assume that  $w=w_2\in \mathbb V_2$. The function $h(z,w_2)$ depends only on the component $z_2$ of $z$ and the restriction $h_{\vert \mathbb V_2\times \mathbb V_2}$ coincide with the complex fondamental kernel $h^{(2)}$ of $\mathbb V_2$. Now $\mathbb E_s$ is a  linear combination of the differential operators $M_k$. By Proposition \ref{restrick2}, we are reduced to a computation inside $\mathbb V_2$. Now use the Bernstein-Sato identity for $\mathbb V_2$ (equation \eqref{BSHJA}) to conclude.
\end{proof}

\section{Bernstein-Sato identities on PJTS  }

\subsection{PJTS with a complex structure}

Let $V$ be a simple PJTS admitting a complex structure, i.e. a $PHJTS$ but viewed as a real PJTS. Let $h(x,y)$ be the complex fundamental kernel of the corresponding PHJTS. Then, as a \emph{real} PJTS, its fundamental kernel is given by
\[k(x,y) = \vert h(x,y)\vert^2\ .\]
It is a polynomial function  which is nonnegative on $V\times V$.

Recall that for $g\in \mathbb G$, and $x$ where $g$ is defined, then $j(g,x)$ is the  Jacobian of the holomorphic map $x\longmapsto g(x)$.  Let  $j_\mathbb R (g,x) = \vert j(g,x)\vert^2$ be its Jacobian on $\mathbb X$ viewed as a real manifold.

 Let $D = \sum_{\alpha} a_\alpha(z) \frac{\partial^{\vert \alpha\vert}}{\partial z^\alpha}$ be a holomorphic differential operator on $V$. Denote by $\overline D$ the (antiholomorphic) differential operator given by
 \[ \overline D = \sum_{j=1} ^k \overline {a_\alpha}(z) \frac{\partial^{\vert \alpha\vert}}{\partial \overline{z}^\alpha}\ .
 \]
 Then, for any smooth function on $V$,  $\overline {D}\,(\overline{f}) = \overline {Df}$.
 
 Let $\mathbb E_s$ be the holomorphic differential operator defined on $V$ (viewed as a PHJTS) by \eqref{EsMk} and let $b(s)=b_{r,a_\mathbb C}(s)$ be defined by \eqref{b}.
 
 \begin{proposition} Let  $x,y\in V$ and assume that $k(x,y)\neq 0$. Then
 \[ \mathbb E_s\, \overline {\mathbb E_s}\, k(x,y)^s = b(s)^2 k(x,y)^{s-1}\ .\]

\end{proposition}

\begin{proof} Near a point $(x,y)$ such that $h(x,y)\neq 0$, choose a local determination of $\log h(x,y)$, and consider the corresponding local determinations for the complex  powers of $h$, such as $h(x,y)^s$.  Now, $\overline{ \big(h(x,y)^{\overline s}\big)}$ is a determination of $\big( \overline {h(x,y)}\big)^s$ such that 
 \[k(x,y)^s : = e^{s\log k(x,y)} = h(x,y)^s\,\overline{ h(x,y)^{\overline s}}\ .\] 
Hence 
\[\big(E_s \,\overline {E_s}\,\big)\, k(x,y)^s = E_s h(x,y)^s\,\, \overline {E_s}\,\overline {h(x,y)^{\overline s}}
\]
\[ = b(s)\,\overline{b(\overline s)} \,h(x,y)^{s-1} \overline {h(x,y)^{\overline s -1}}
= b(s)^2\, k(x,y)^{s-1}\ .\]
 
\end{proof}

\subsection{Reduced PJTS}

Let $V$ be a simple reduced PJTS, and let $\mathbb V$ its Hermitification, which is a simple PHJTS. Let $h$ be the complex fundamental kernel of $\mathbb V$. Recall that the fundamental kernel $k$ of $V$ is given by
\[ k(x,y) = h(x,y)^2\ .\]
Then $k$ is a polynomial function which is nonegative on $V\times V$.

With again \eqref{EsMk} in mind, let $E_s$ be the differential operator on $V$ defined by
\[ E_s =  \sum_{k=0}^r (-1)^{r-k} \big(\prod_{j=1}^k(s+\frac{d}{2}(r-j)\big) M_{r-k}\ .
\]
Then the complexified holomorphic differential operator on $\mathbb V$ coincides with the holomorphic differential operator denoted by $\mathbb E_s$.

Let $t\in \mathbb R, t\neq 0$. For $t>0$, we choose the usual determination of $\log t$ and the corresponding determination of $t^s$. When $t<0$, we use $\log t = \log (-t) +i\pi$, so that the corresponding determinaton of  $t^s$ is $t^s = (-t)^s e^{i\pi s}$. With these notations, observe that
$t^{2s}={(t^2)}^s$ for $t>0$, and $t^{2s} = e^{2i\pi s}(t^2)^s$ when $t<0$.

\begin{proposition} Let $k(x,y)\neq 0$. Then

\[\big( E_{2s-1} \circ E_{2s}\big)k(x,y)^s = b(2s)b(2s-1) k(x,y)^{s-1}\ .\]
\end{proposition}
\begin{proof}

By applying twice the Bernstein-Sato identity on $\mathbb V$ for the values $2s$ and $2s-1$, we get
\[\big( \mathbb E_{2s-1} \circ \mathbb E_{2s}\big)h(x,y)^{2s} = b(2s)b(2s-1) h(x,y)^{2s-2}\ ,\]
where $b=b_{r,a}$ ($r=\rank V$ and  $a$ the first characteristic number of $V$, equal to the rank and to the  first complex characteristic number of $\mathbb  V$). Now  restrict this identity to $V$ by the restriction principle. When $h(x,y)>0$, this clearly implies the statement. When $h(x,y)<0$, oberve that $h(x,y)^{2s} = e^{-2i\pi s} k(x,y)^2$, whereas $h(x,y)^{2s-2} = e^{-2i\pi(s-1)} k(x,y)^s$, so that the statement also holds.
\end{proof}

\subsection{The non reduced case}

Let $V$ be a simple PJTS without complex structure and assume that it is non-reduced. Let $\mathbb V$ be its Hermitification. Let $h(x,y)$ be the complex fundamental kernel of $\mathbb V$. Recall that the fundamental kernel $k$ of $V$ is defined by $k(x,y)=h(x,y)$. As observed in Proposition \ref{posNR}, the kernel $k(x,y)$ is everywhere $\geq 0$. 

\begin{proposition} Let $x,y$ in $V$ and assume that $k(x,y)\neq 0$. Then
\[E_s k(x,y)^s = b(s) k(x,y)^{s-1}\ .
\]
where 
$ b= b_{2r,c-2}$.
\end{proposition}

\begin{proof} Observe that, for $y\in V$,  $h(x,y)$ is a holomorphic function of $x\in \mathbb V$, which is real-valued on $V$. Hence we may apply the restriction rule \eqref{complexD} to the Berstein-Sato identity \eqref{BSHJA} on $\mathbb V$. As the rank of $\mathbb V$ is $2r$ and  $ a_\mathbb C=c-2$, the polynomial $b$ is $b_{2r,a_\mathbb C} = b_{2r, c-2}$.
\end{proof}

\section{Meromorphic continuation of the intertwining operators}

Let $f$ be in $\mathcal C^\infty(X)$. The intertwining operator is given by
\[J_\lambda f (s) = \int_X \widetilde c (s,t)^{-\frac{1}{2} +\lambda} f(t) d \sigma(t)\ .
\]

We want to prove that $\lambda \longmapsto J_\lambda(f)(x)$ can be continued meromorphically to $\mathbb C$.

Observe that , for any $k\in K$ $\rho_\lambda(k)$  is the left translation by $k$, and similarly for $\rho_{-\lambda}^\sigma(k)$.  So $J_\lambda$ commutes to the left translation by elements of $K$. Hence the meromorphic continuation, if valid for a function $f$, is also valid for any $K$-translate of $f$. By a partition of unity argument, any function $f$ in $\mathcal C^\infty(X)$ can be written as a finite sum of $K$-translates of functions which have their supports contained in $\mathcal O$. So it is enough to prove the meromorphic continuation for functions having their support contained in $\mathcal O$. This remark allows to transfer the problem of meromorphic continuation from the compact picture to the noncompact picture.

Denote by $\mathcal C^\infty(X,\mathcal O)$ the space of smooth functions which have their support contained in $\mathcal O$.
Let $\iota_\lambda$ be the map defined on $\mathcal C^\infty(X,\mathcal O)$ by
\[\iota_\lambda f (x) = \big(c(x,x)^{-\frac{1}{2}}\big)^{\frac{1}{2}+\lambda} f(\kappa(x))\ .
\]
in accordance with the way densities should transform under the variable change $s=\kappa(x)$.

\begin{proposition} For any $f\in \mathcal C^\infty(X,\mathcal O)$,

\begin{equation} \iota_{-\lambda} (J_\lambda f) (x) = \int_V c(x,y)^{-\frac{1}{2} + \lambda} (\iota_\lambda f) (y) dy\ .
\end{equation}
\end{proposition}

\begin{proof} 
\[ \iota_{-\lambda}(J_\lambda f)(x) = \big(c(x,x)^{-\frac{1}{2}}\big)^{\frac{1}{2}-\lambda}
\int_S \widetilde c (\kappa(x), t)^{-\frac{1}{2}+\lambda} f(t) d\sigma(t)\]
\[ =\big(c(x,x)^{-\frac{1}{2}}\big)^{\frac{1}{2}-\lambda}
\int_V \widetilde c (\kappa(x), \kappa(y))^{-\frac{1}{2}+\lambda}f(\kappa(y)) c(y,y)^{-\frac{1}{2}} dy
\]
\[ =\big(c(x,x)^{\frac{1}{2}}\big)^{-\frac{1}{2}+\lambda}
\int_V \widetilde c (\kappa(x), \kappa(y))^{-\frac{1}{2}+\lambda}(\iota_\lambda f)(y) \big(c(y,y)^{\frac{1}{2}}\big)^{-\frac{1}{2} +\lambda} dy
\]
\[ = \int_V c(x,y)^{-\frac{1}{2} +\lambda} (\iota_\lambda f)(y) dy\ ,
\]
where we used the change of variable $t=\kappa(y)$ first, and then \eqref{ctildec}.
\end{proof}

As the map $\iota_\lambda f$ depends holomorphically on $\lambda$ on $\mathbb C$,  the meromorphic continuation of $J_\lambda f$ can be deduced, by a routine argument from the meromorphic continuation of 
\[I_\lambda f(x) = \int_V c(x,y)^{-\frac{1}{2}+\lambda} f(y) dy
\]
where $f$ is a function in $\mathbb C_c^\infty(V)$.

In order to describe the meromorphic continuation, it will be easier to work with the analytic continuation in $s$ of the integral

\[J_s = J_s(f) = \int_V k(x,y)^s f(y) dy\ ,\]

the two problems being related by the equality
\[I_\lambda f = J_s f, \quad \text {for } s = -\frac{p}{4}+\frac{p\lambda}{2}\ .
\]

By theorem \ref{absconv}, the integral $J_s$ is absolutely convergent if $\Re s > -\frac{c}{2}$. Now the various Bernstein identities may be written as 
\[B_s(x,\frac{\partial}{\partial x}) k(x,y)^s = b(s) k(x,y)^{s-1}\]
where $B_s$ is a differential operator with polynomial coefficients, and $b$ is  polynomial in $s$. Hence, by an integration by parts,
\[J_{s-1} (f)= b(s)^{-1} \int_V k(x,y)^s (B_s^t f) (y) dy=b(s)^{-1} J_s(B_s^t f) .
\]
This can be used to define the meromorphic continuation near a point $s-1$ provided $s$ is not a pole of $J_s$ and  $b(s)\neq 0$. 

For $s=-\frac{c}{2}$, the integral $J_s1$ studied in section 7 is not convergent. As the integrand is a nonnegative function, it really means that  $-\frac{c}{2}$ is a pole of $J_s$. Then $1-\frac{2}{c}$ must be a zero of $b$. The set of poles is  the union of  the "left integral half lines" $\{\alpha- k, k=1,2\dots\}$ where $\alpha$ runs through  the zeroes of $b$ such that $\alpha -1- \leq \frac{c}{2}$. 

\subsection{ Meromorphic continuation for a PHJTS}

Let $V$ be a simple PHJTS, which is regarded as a simple PJTS admitting a complex structure. The constant $c$ in this case is equal to $2$, so that the integral $I_s$ is absolutely convergent if $\Re s > -1$. Let $a$ be the complex dimension of $V_{ij}$ for $i\neq j$.

\begin{theorem} Let $V$ be a simple PHJTS. The integral $J_s$ has a meromorphic extension to $\mathbb C$ with poles at
 
$\bullet$ $-1,-2,\dots, -k,\dots$ if $r=1$ or if $r\geq 2$ and $a$ is even.

$\bullet$ $-1,-2,\dots, -k,\dots$ and $-1-\frac{a}{2}, \dots, -k-\frac{a}{2}$, if  $a$ is odd.
\end{theorem}

\begin{proof} The Bernstein polynomial $b$ is given by  
 $b(s) =b_{r,a}(s)^2$ and its set of zeroes is
   \[\{0,-\frac{a}{2},\dots, -(r-1)\frac{a}{2}\}\ .\]  The result follows.
\end{proof}

\subsection{Meromorphic continuation for a reduced PJTS}

\begin{theorem} Let $V$ be a simple reduced PJTS. The integral $J_s$ has a meromorphic extension to $\mathbb C$ with poles at

$\bullet$ $-\frac{1}{2},-1,-\frac{3}{2},\dots, -\frac{k}{2},\dots$ if $r=1$ or if $r\geq 2$ and $a$ is even.

$\bullet$ $-\frac{1}{2},-1,-\frac{3}{2},\dots, -\frac{k}{2},\dots$  and $-\frac{a}{4}-\frac{1}{2},\dots, -\frac{a}{4}-\frac{k}{2},\dots$ if $a$ is odd.
\end{theorem}

\begin{proof}
In this case, $c=1$, so that the integral $I_s$ converges absolutely for $\Re s > -\frac{1}{2}$. The Bernstein polynomial $b$ is given by \[b(s)  = b_{r,a}(2s)b_{r,a}(2s-1)  \]

The zeroes of $b(2s)b(2s-1)$ are are \[ \{0, -\frac{a}{4},\dots -(r-1)\frac{a}{4}\}\cup\{ \frac{1}{2},\frac{1}{2}-\frac{a}{4},\dots,\frac{1}{2} -(r-1)\frac{a}{4}\}\ .\]
The result follows.
\end{proof}

\subsection{ The non reduced case}

\begin{theorem} Let $V$ be a simple PJTS without complex structure and non reduced. Then the integral $J_s$ has a meromorphic extension to $\mathbb C$ with poles at

$\bullet$ $-\frac{c}{2},-\frac{c}{2}-1,\dots, -\frac{c}{2} -k,\dots$ if $r=1$ or if $r\geq 2$ and $c$ is even

$\bullet$ $-\frac{c}{2},-\frac{c}{2}-1,\dots, -\frac{c}{2} -k,\dots$ and $-c+1,-c,-c-1,\dots, -c-k,\dots$ if $r\geq 2$ and $c$ is odd.

\end{theorem}

\begin{proof}
In this case, the first pole is at $-\frac{c}{2}$, and the polynomial $b$ is equal to $b_{2r,c-2}$ which has for zeroes
\[ \{0,-\frac{c}{2}+1,\dots,(2r-1) (-\frac{c}{2}+1\}\ ,
\]
As $c\geq 3$, $-1>-\frac{c}{2}$ does not match the condition to be a pole. The second zero of $b$  corresponds to the first pole and produces the first family of poles. If $r=1$ this exhausts the possibilites, but  if $r\geq 2$ and $c$ is odd, the third zero of $b$ contributes to a second family of poles.
\end{proof}

\medskip
\footnotesize{\noindent Address\\ Jean-Louis Clerc\\Institut Elie Cartan, Universit\'e de Lorraine, 54506 Vand\oe uvre-l\`es-Nancy, France\\}
\medskip
\noindent \texttt{{jean-louis.clerc@univ-lorraine.fr
}}

\end{document}